\documentclass[11pt]{article}
\usepackage{amsmath}
\usepackage{amssymb}
\usepackage{amsthm}
\usepackage{comment}
\usepackage[morefloats=200]{morefloats}
\usepackage{enumitem}
\setlist{nosep}

\newcommand{\Z}{{\mathbb Z}}
\newtheorem{theorem}{Theorem}[section]

\newtheorem{lemma}[theorem]{Lemma}

\newtheorem{construction}[theorem]{Construction}

\newtheorem{corollary}[theorem]{Corollary}

\begin{document}
\baselineskip 18pt
\title{Group divisible designs with block size $4$ and group sizes $4$ and $7$}

{\small
\author
{R. Julian R. Abel \\
 School of Mathematics and Statistics\\
  UNSW Sydney\\ NSW 2052, Australia\\
 \texttt{r.j.abel@unsw.edu.au }
 \and
   Thomas Britz  \\
  School of Mathematics and Statistics\\
   UNSW Sydney\\ NSW 2052, Australia\\
  \texttt{britz@unsw.edu.au  }
 \and
   Yudhistira A.  Bunjamin  \\
  School of Mathematics and Statistics\\
   UNSW Sydney\\ NSW 2052, Australia\\
  \texttt{yudhi@unsw.edu.au  }
  \and
  Diana Combe  \\
  School of Mathematics and Statistics\\
   UNSW Sydney\\ NSW 2052, Australia\\
  \texttt{diana@unsw.edu.au }
}    
}    

\date{}

\maketitle 

\noindent{\bf Abstract:} 
In this paper, we consider the existence of group divisible designs (GDDs) with block size $4$ and group sizes $4$ and $7$.
We show that there exists a $4$-GDD of type $4^t 7^s$ for all but a finite specified set of feasible values for $(t,s)$.

\noindent{\bf Keywords:} group divisible design (GDD), feasible group type.

\noindent{\bf Mathematics Subject Classification:}  05B05

\section{Introduction}

Let $X$ be a finite set of {\it points} with a partition into parts which we call {\it groups}. 
Any $k$-element subset of $X$ is called a {\it block}. 
A collection of {\it blocks} is a {\it group divisible design} with block size $k$, or $k$-GDD if 
(i) no two points from the same group appear together in any block and 
(ii) any two points from distinct groups appear together in exactly one block. 
The {\it group type} (or {\it type}) of a $k$-GDD is the multiset $\left\{|G|: G {\text{ is a part of the partition}}\right\}$. 
The group type can also be expressed in  `exponential' notation where the type $t_1^{u_1}  t_2^{u_2} \ldots t_m^{u_m}$ means there are $u_i$ groups of size $t_i$ for $i=1,2, \dots, m$. 
A $k$-GDD of type $g^k$, in which there are $k$ groups all of the same size $g$, is commonly referred to as a {\it transversal design}, denoted TD$(k,g)$.

There are known necessary conditions for the existence of a $4$-GDD of type $\{g_1, g_2, \ldots, g_m\}$.
These are given in Theorem~\ref{necessary}.  
These necessary conditions are not sufficient. 

\begin{theorem}{\rm\cite{ABC.30less, ABC.50less, krestin, reesk=45}} \label{necessary}
Suppose there exists a $4$-GDD of type $\{g_1, g_2, \ldots, g_m\}$ where $g_1 \geq g_2 \geq \cdots \geq g_m >0$. 
Set $v= \sum_{i=1}^m g_i$.  
Then
 \begin{enumerate}
  \item $m \geq 4;$
  \item $v \equiv g_1 \equiv g_2 \equiv \cdots \equiv g_m \pmod{3};$
  \item $\sum_{i=1}^m g_i(v-g_i)  \equiv 0\pmod{4};$
  \item $3g_i + g_j \leq v$ for all $i,j \in \{1,2, \ldots, m\}$, $i \neq j;$
  \item if $m=4$, then $g_i = g_j$ for all $i,j \in \{1,2, \ldots, m\};$
  \item if $m=5$, then the group type is of the form $h^4 n^1$ where $n \leq 3h/2;$
  \item if $3g_1+g_2=v$  and   $g_1  >  g_2$, then $g_3 \leq 2g_1/3;$ 
  \item if the group type is of the form  $h_1^1$ $h_2^x$ $h_3^1$ $h_4^1  \cdots h_n^1$ 
 where $x \geq 1$ and $3h_1 + h_2 = v$ and $h_2 > h_3 \geq \cdots  \geq h_n$, then $n \geq 6$.
If further $n=6$,  then for $i,j \in \{3,4,5,6\}$  we have $h_i (h_2-h_i)  = h_j(h_2 - h_j)$. 
 \end{enumerate}
\end{theorem}

In this paper we are concerned with $4$-GDDs. 
We say that a multiset $\{g_1, g_2, \ldots, g_m\}$ of positive integers is a {\it feasible} group type for a $4$-GDD if it satisfies the conditions of Theorem~\ref{necessary}. 
Much work on $4$-GDDs has concentrated on the existence question, determining whether there exists a GDD with a particular type - sometimes looking at GDDs on small sets, sometimes looking at infinite families. 
Some work on $4$-GDDs has considered the enumeration question - for a particular type determining how many nonisomorphic designs can be constructed and comparing their automorphism groups.

It is sometimes convenient to consider alternative forms of some of the necessary conditions in Theorem~\ref{necessary}.
In Lemma~\ref{lemma.4gdd.necessary.condition.alternative.group.sizes.mod3} and \ref{lemma.4gdd.number.of.odd.sized.groups}, we show the equivalence of some simpler conditions that are particularly useful in this paper.

\begin{lemma}
\label{lemma.4gdd.necessary.condition.alternative.group.sizes.mod3}
In the necessary conditions for the existence of a $4$-GDD of type $\{g_1, g_2, \ldots, g_m\}$, the condition that $v \equiv g_1 \equiv g_2 \equiv \cdots \equiv g_m \pmod{3}$ can be replaced by the condition that $g_1 \equiv g_2 \equiv \cdots \equiv g_m \pmod{3}$ and either $g_1 \equiv 0 \pmod{3}$ or $m \equiv 1 \pmod{3}$.

\begin{proof}
Firstly, suppose that $v \equiv g_1 \equiv g_2 \equiv \cdots \equiv g_m \pmod{3}$.
Recall that $v = g_1 + g_2 + \cdots + g_m$.
Thus, $v \equiv g_1 + g_2 + \cdots + g_m \equiv mg_1 \pmod{3}$.
This means that $v \equiv g_1 \equiv mg_1 \pmod{3}$.
Hence, either $g_1 \equiv 0 \pmod{3}$ or $m \equiv 1 \pmod{3}$.

Next, suppose that $g_1 \equiv g_2 \equiv \cdots \equiv g_m \pmod{3}$ and $g_1 \equiv 0 \pmod{3}$. 
Then $g_1 \equiv g_2 \equiv \cdots \equiv g_m \equiv 0 \pmod{3}$.
This means that $v \equiv g_1 + g_2 + \cdots + g_m \equiv 0 \pmod{3}$.
Hence, $v \equiv g_1 \equiv g_2 \equiv \cdots \equiv g_m \pmod{3}$.

Finally, suppose that $g_1 \equiv g_2 \equiv \cdots \equiv g_m \pmod{3}$ and $m \equiv 1 \pmod{3}$.
Then $v \equiv g_1 + g_2 + \cdots + g_m \equiv mg_1 \equiv g_1 \pmod{3}$.
Hence, $v \equiv g_1 \equiv g_2 \equiv \cdots \equiv g_m \pmod{3}$.
\end{proof}
\end{lemma}

\begin{lemma} \label{lemma.4gdd.number.of.odd.sized.groups}
In the necessary conditions for the existence of a $4$-GDD of type $\{g_1, g_2, \ldots, g_m\}$,
the condition that $\sum_{i=1}^m  g_i(v-g_i) \equiv 0 \pmod{4}$ 
can be replaced by the condition that the number of groups of odd size is either $0$ or $1$ modulo $4$.

\begin{proof}
Consider the sum of point-block pairs $S = \sum_{i=1}^m g_i(v-g_i)$ and suppose that 
\[ n = \left|\left\{i \in \{1,\ldots,m\} \::\: g_i \equiv 1 \!\!\! \pmod{2} \right\}\right|. \]
Then $n$ is the number of groups of odd size.
Note that $v \equiv n \pmod{2}$.

Suppose that $v$ is even.
If $g_i \equiv 0 \pmod{2}$, 
then $g_i(v-g_i) \equiv 0 \pmod{4}$.
If $g_i \equiv 1 \pmod{2}$, 
then $g_i(v-g_i) \equiv 1 \pmod{4}$ when $v \equiv 2 \pmod{4}$ 
and  $g_i(v-g_i) \equiv-1 \pmod{4}$ when $v \equiv 0 \pmod{4}$.
Thus,
$S \equiv \pm n \pmod{4}$.
Hence, when $v$ is even, $S \equiv 0 \pmod{4}$ if and only if $n \equiv 0 \pmod{4}$.

Now, suppose that $v$ is odd.
If $a$, $b$ and $c$ are the numbers of groups of sizes 
congruent to $v$, $2$ and $v+2$ modulo $4$ respectively,
then $n = a+c$.
If   $g_i\equiv 0$ or $v \pmod{4}$, 
then $g_i(v-g_i)\equiv 0 \pmod{4}$.
If   $g_i\equiv 2$ or $v+2\pmod{4}$, 
then $g_i(v-g_i)\equiv 2\pmod{4}$.
Thus, $S\equiv 2(b+c)\pmod{4}$.
Next, counting the number of points in the $4$-GDD modulo $4$ gives $v\equiv va+2b+(v+2)c \pmod{4}$ so rearranging gives $2(b+c) \equiv v(1-a-c) \pmod{4}$.
Thus, $S \equiv 2(b+c) \equiv v(1-a-c) \equiv v(1-n) \pmod{4}$.
Hence, when $v$ is odd, $S \equiv 0 \pmod{4}$ if and only if $n \equiv 1 \pmod{4}$.
\end{proof}

\end{lemma}

The main aim of this paper is to answer the existence question for $4$-GDDs of type $4^t 7^s$.
This paper is organised as follows. 
In Section~\ref{s:known.results}, we state some prior-known results regarding $4$-GDDs, especially those on $4$-GDDs with up to $50$ points and $4$-GDDs of type $g^p$ and $g^p n^1$ which will be used later on.
In Section~\ref{s:4t7s.constructions}, we give constructions for all feasible types of $4^t 7^s$ except for a short and finite list of exceptions.
Finally, in Section~\ref{s:summary}, we state the main result of this paper and provide a corollary to this result.

\section{Some known results about $4$-GDDs}
\label{s:known.results}

\subsection{Designs on sets with $v \le 50$}
The existence of $4$-GDDs with no more than $30$ points has been well investigated. 
In~\cite{krestin},  the existence question was answered for all but three types,
 and the existence results for those three types were completed in ~\cite{ABBC.2t5s, ABC.30less}.  

\begin{theorem} {\rm\cite {ABBC.2t5s, ABC.30less, krestin}} \label{gddlt30}
The feasible group types for a $4$-GDD on at most $30$ points are 
$1^4$,~$2^4$, $3^4$, $1^{13}$, $1^{9}4^1$, $2^7$, $3^5$,
$1^{16}$, $1^{12} 4^1$, $1^{8} 4^2$, $1^{4} 4^3$, $4^4$, $2^6 5^1$, $2^{10}$, $5^4$, $3^5 6^1$,
$1^{15} 7^1$, $2^9 5^1$, $3^8$, $3^4 6^2$, $6^4$, $1^{25}$, $1^{21} 4^1$, $1^{17} 4^2$, $1^{13}4^3$,
$1^9 4^4$, $1^5 4^5$, $1^1 4^6$, $2^{13}$, $2^3 5^4$, $2^9 8^1$, $3^9$, $3^5 6^2$, $3^1 6^4$,
$1^{28}$, $1^{24} 4^1$, $1^{20} 4^2$, $1^{16}4^3$, $1^{12} 4^4$,
$1^8 4^5$, $1^4 4^6$, $4^7$, $1^{14} 7^2$, $1^{10} 4^1 7^2$, $1^6 4^2 7^2$,
$1^2 4^3 7^2$, $7^4$, $2^{12} 5^1$, $2^2 5^5$, $2^8 5^1 8^1$, $3^8 6^1$, $3^4 6^3$, $6^5$, $3^7 9^1$.  
A~$4$-GDD exists for all these types with the definite exception of types $2^4$, $2^6 5^1$ and $6^4$. 
\end{theorem}

Existence results have been extended to point sets with up to $v=50$ points. 
For $31 \leq v \leq 50$ and $v \equiv 0 \pmod{3}$ there exist designs for all feasible types; this result was completed in~\cite{ABC.50less}. 
In~\cite{ABBCF.1mod3}, the feasible types for $31 \leq v \leq 50$ and $v \equiv 1,2 \pmod{3}$ are considered.  
The questions of existence were completed for $v \equiv 1 \pmod{3}$; and for $v \equiv 2 \pmod{3}$ the known results were extended leaving unknown the question of existence for types $2^{11} 8^1 11^1$, $2^1 5^4 8^1 11^1$, $2^6 5^2 11^2$, 
$2^{5} 5^3 8^1 11^1$, 
$2^{2} 5^2 8^1 11^2$,  $2^{1} 5^3 8^2 11^1$, 
 $2^{5} 5^3 11^2$, $2^{2} 5^2 11^3$, $2^{1} 5^3 8^1 11^2$, $5^{4} 8^2 11^1$,
  $2^{9} 5^2 11^2$, 
 $2^{8} 5^3 8^1 11^1$,   $2^6 5^1 11^3$, $2^{5} 5^2 8^1 11^2$,  $2^{4} 5^4 8^1 14^1$, $2^{4} 5^3 8^2 11^1$,  $2^{3} 11^4$,  
$2^{2} 5^7 11^1$,  $2^{2} 5^1 8^1 11^3$,  $2^{1} 5^2 8^2 11^2$ and   $5^{3} 8^3 11^1$.

\subsection{Existence results for infinite families of types} \label{knowngdds}

Extensive work has been done on the existence of $4$-GDDs of types $g^p$ and $g^p n^1$. 
For type $g^p$, existence has been completely determined as given in Lemma~\ref{lemma.gp}.

\begin{lemma}{\rm\cite {BSH}}
\label{lemma.gp} \label{lemma.4gdd.g^p}
There exists a $4$-GDD of type $g^p$  if and only if  
$(1)$  $p \geq 4;$ $(2)$  $g(p-1) \equiv 0 \pmod{3};$  $(3)$  $g^2p(p-1) \equiv 0 \pmod{12}$ and $(4)$  $(g,p) \notin \{(2,4), (6,4)\}$.
\end{lemma}

For type $g^p n^1$ with $p \geq 4$, necessary conditions for existence are $g \equiv n\pmod{3}$, $gp \equiv 0\pmod{3}$, $gp\big(g(p-1) + 2n\big) \equiv 0\pmod{4}$ and $0 \leq n \leq g(p-1)/2$.
The existence of such designs has been determined in many cases;
see for instance Forbes~\cite{Forbes2,Forbes3, Forbes1}, Ge and Rees~\cite{gerees}, Ge et al.~\cite{gezhu}, Rees~\cite{reesk=45}, and Wei and Ge~\cite{gegum}.  
However, for a number of these feasible types, the question of the existence of a corresponding design remains unanswered.

In this paper we make use of the existence results given in Lemmas~\ref{lemma.gpn1.0mod6} to \ref{lemma.4gdd.28p_n1}.

\begin{lemma} {\rm~\cite{Forbes1, WeiGe.4gdd.gpn1.g0mod6, gegum}} 
\label{gpn1.0mod6}
\label{lemma.gpn1.0mod6}
Suppose that $g \equiv 0\pmod{6},$ $g \geq 6,$ $p \geq 4,$ $n \equiv 0 \pmod{3}$ and $0 \leq n \leq g(p-1)/2$.
Then there exists a $4$-GDD of type $g^p n^1$ except when $(g,p,n) = (6,4,0)$.
\end{lemma}

\begin{lemma}{\rm\cite{Forbes2, gegum}} 
\label{3pn1}
\label{lemma.gpn1.g3mod6}
Suppose that $g \equiv 3 \pmod{6},$ $g < 141,$ $p \geq 4$ and one of the following conditions hold$:$ 
\begin{itemize}
    \item $p \equiv 0\pmod{4}$ and $n \equiv 0\pmod{3}$ where $0 \leq n \leq (g(p-1)-3)/2,$ or
    \item $p \equiv 1\pmod{4}$ and $n \equiv 0\pmod{6}$ where $0 \leq n \leq g(p-1)/2,$ or
    \item $p \equiv 3\pmod{4}$ and $n \equiv 3\pmod{6}$ where $3 \leq n \leq g(p-1)/2$.
\end{itemize}
Then there exists a $4$-GDD of type $g^p n^1$.
\end{lemma}

\begin{lemma}{\rm \cite{GeLing.2004.gum1}} 
\label{lemma.4pn1}
There exists a $4$-GDD of type~$4^p n^1$ if $p \geq 4,$ $p \equiv 0 \pmod{3},$ $n \equiv 1 \pmod{3}$ and $1 \leq n \leq 2(p-1)$.

\begin{proof}
If $p \in \{6, 9\}$, then the $4$-GDD exists by \cite[Lemma~3.17]{GeLing.2004.gum1}.
If $p \geq 12$, then the $4$-GDD exists by \cite[Theorem~3.18]{GeLing.2004.gum1}.
\end{proof}
\end{lemma}

\begin{lemma}{\rm\cite{gezhu, schuster.2014.gum1, WeiGe.2013.more.gum1}} 
\label{lemma.7pn1}
There exists a $4$-GDD of type~$7^p n^1$ if $p \geq 4$ and one of the following conditions hold$:$
\begin{itemize}
    \item $p \equiv 0 \pmod{12}$ and $n \equiv 1 \pmod{3}$ where $1 \leq n \leq (7(p-1)-3)/2,$ or
    \item $p \equiv 3 \pmod{12}$ and $n \equiv 1 \pmod{6}$ where $1 \leq n \leq 7(p-1)/2,$ or
    \item $p \equiv 9 \pmod{12}$ and $n \equiv 4 \pmod{6}$ where $4 \leq n \leq 7(p-1)/2$.
\end{itemize}

\begin{proof}
If $p = 9$, then $n\in \{4,10,16,22,28\}$, so the $4$-GDD exists by \cite[Lemma~2.18]{schuster.2014.gum1} or \cite[Lemma~3.14]{WeiGe.2013.more.gum1}.
If $p \geq 12$ and $n \in \{1,4\}$, then the $4$-GDD exists by \cite[Theorem~2.11]{gezhu}.
If $p \geq 12$ and $n \geq 7$, then the $4$-GDD exists by \cite[Theorem~3.15]{WeiGe.2013.more.gum1}.
\end{proof}
\end{lemma}

\begin{lemma}{\rm\cite{Forbes3, schuster.2014.gum1}}
\label{lemma.4gdd.28p_n1}
There exists a $4$-GDD of type~$28^p n^1$ if $p \geq 4,$ $p \equiv 0 \pmod{3},$ $n \equiv 1 \pmod{3}$ and $n \leq 14(p-1)$.

\begin{proof}
Constructions for types~$28^9 19^1$, $28^9 25^1$ and~$28^9 31^1$ are given in \cite[Lemma~2.6]{Forbes3}.
Otherwise, they exist by \cite[Theorem~5.17]{schuster.2014.gum1}.
\end{proof}
\end{lemma}


Other work on the existence of $4$-GDDs has concentrated on whether the group sizes are congruent to $0$, $1$ or $2 \pmod{3}$ and on GDDs whose groups are of only two or three different sizes. 
See, for example,~\cite{ABBC.2t8s,ABBC.2t5s,ABBCF.1mod3,ABC.30less,ABC.50less,ABC.3562,BSH,Forbes3,Forbesgcc,gegum}.

\subsection{Some known enumeration results}

In \cite{kreher.210note}, the $4$-GDDs of type~$2^{10}$ were enumerated; up to isomorphism there are five different such $4$-GDDs and each has a non-trivial automorphism group. 
The automorphism groups have orders~$8$, $12$, $16$, $72$ and~$720$.
The $4$-GDDs of type $7^4$ were enumerated in \cite{ABBCF.1mod3}; up to isomorphism there are seven different such $4$-GDDs and, as in the case of the designs of type~$2^{10}$, each has a non-trivial automorphism group. Two of the $4$-GDDs of type $7^4$ have automorphism groups of order~$6$ and the remaining ones have orders~$24$, $48$, $126$, $2352$ and~$3528$.  
In~\cite{ABC.3562}, the $4$-GDDs of type~$3^5 6^2$ were enumerated and only~$3$ out of~$22$ had any non-trivial automorphisms.
Of those that did have non-trivial automorphisms, one had an automorphism group of order~$2$ and two had automorphism groups of order~$3$.

\subsection{Some direct constructions}
\label{s:some.direct.constructions}
In the past, several GDDs have been found directly by assuming the existence of a cyclic automorphism group, $\mathbb{G}$. 
See  for instance~\cite{ABBC.2t8s,ABBC.2t5s,ABBCF.1mod3,ABC.30less,ABC.50less,ABC.3562,Forbes2,Forbes3,Forbesgcc,Forbes1,WeiGe.4gdd.gpn1.g0mod6,gegum}.  
In these designs, one assigns a selection of base blocks and then develops these base blocks to produce the blocks of the required design.
For the $4$-GDDs constructed directly in this paper, the point set of the design consists of one or more copies of the group~$\mathbb{G}$
(here, $\mathbb{G}$ is $\mathbb{Z}_5$, $\mathbb{Z}_6$, $\mathbb{Z}_7$ or $\mathbb{Z}_{14}$)
plus possibly a few copies of $\mathbb{Z}_2$ or $\mathbb{Z}_3$ and possibly one or more infinite points.  
Blocks in these designs are obtained by developing the subscripts of the points from copies of $\mathbb{G}$ over $\mathbb{G}$;
the infinite points remain unaltered when developed.  
When the point set includes any extra copies of $\mathbb{Z}_2$ or $\mathbb{Z}_3$, those points are developed over $\mathbb{Z}_2$ or $\mathbb{Z}_3$ as the others are developed over~$\mathbb{G}$.
Also, there are usually a few blocks which remain invariant when some nonzero element of $\mathbb{G}$ is added to it;
the number of blocks generated by any one of those blocks is less than the size of $\mathbb{G}$.

As an example, we give a $4$-GDD of type~$1^3 4^1 7^6$ from
\cite[Table~27]{ABBCF.1mod3} using this method.
The points are  
$a_i,b_i,c_i,d_i,e_i,f_i,g_i$ for $i \in \mathbb{Z}_6$, $y_i$ for $i \in \mathbb{Z}_3$, and $\infty_1,$ $\infty_2,$ $\infty_3,$ $\infty_4$.  
The groups are  $\{a_i,b_i,c_i,d_i,e_i,f_i,g_i\}$ for $i \in \mathbb{Z}_6$, $\{y_i\}$  for $i \in \mathbb{Z}_3$ and
  $\{\infty_1, \infty_2,\infty_3,\infty_4\}$.
The blocks are obtained by developing the blocks in the following array$\pmod{6}$. 
The six blocks in the first column generate three blocks each.
\[\small\begin{array}{|l|l|l|l|l|l|}\hline
    \{a_0,a_3,y_0,\infty_1\}  & \{a_0,a_1,b_2,e_3\}      & \{a_0,c_4,f_3,f_5\}      & \{b_0,b_1,d_3,f_2\}      & \{c_0,c_1,f_3,y_2\}      & \{d_0,f_2,f_3,g_1\}\\\hline
    \{b_0,b_3,y_0,\infty_2\}  & \{a_0,a_2,c_3,e_1\}      & \{a_0,d_5,g_2,y_1\}      & \{b_0,c_4,d_5,e_3\}      & \{c_0,c_2,g_4,g_5\}      & \{e_0,e_1,g_3,g_5\}\\\hline
    \{c_0,c_3,y_0,\infty_3\}  & \{a_0,b_3,c_2,d_1\}      & \{a_0,d_4,g_3,\infty_2\} & \{b_0,c_2,e_5,\infty_4\} & \{c_0,d_2,e_1,\infty_1\} & \\\hline
    \{d_0,d_3,y_0,\infty_4\}  & \{a_0,b_5,f_2,g_4\}      & \{a_0,e_4,f_1,y_2\}      & \{b_0,d_1,f_5,y_2\}      & \{c_0,e_2,f_4,\infty_2\} & \\\hline
    \{e_0,e_3,f_1,f_4\}       & \{a_0,b_4,g_1,\infty_3\} & \{a_0,f_4,g_5,\infty_4\} & \{b_0,e_2,e_4,y_1\}      & \{d_0,d_2,e_3,g_4\}      & \\\hline
    \{g_0,g_3,y_0,y_1\}       & \{a_0,c_5,d_2,d_3\}      & \{b_0,b_2,c_3,g_4\}      & \{b_0,f_4,g_1,\infty_1\} & \{d_0,e_2,f_1,\infty_3\} & \\\hline
  \end{array}\]

All $4$-GDDs that are constructed directly in this paper are given in the Appendix.

\section{Constructions of $4$-GDDs of type~$4^t 7^s$}
\label{s:4t7s.constructions}

One tool for constructing GDDs from other GDDs is that of `filling in groups'.  
This is given in Construction~\ref{construction.fillin} in the form applicable to $k$-GDDs.  
This construction is valid by~\cite[Theorem 1.4.12]{fmy}; see also the proof of that theorem.

\begin{construction} \label{construction.fillin}
Suppose that a $k$-GDD of group type $\{g_1, g_2, \ldots, g_m\}$ exists and let $u$ be a non-negative integer.
Suppose also that, for each integer $i=1,2, \ldots, m-1$, there exists a $k$-GDD of type
$\{g(i,1),g(i,2),\ldots,g(i,s_i),u\}$ for which $\sum_{j=1}^{s_i} g(i,j)=g_i$.
Then there exists a $k$-GDD of type $\{g(1,1), g(1,2),\ldots, g(1,s_1), g(2,1), g(2,2),\ldots, g(2,s_2),\ldots, g(m-1,1), g(m-1,2),\ldots, g(m-1,s_{m-1}), g_m + u\}$.
If further, there exists a $k$-GDD on $g_m + u$ points of type  $\{g(m,1), g(m,2), \ldots,$ $g(m,s_m)\}$,
then there exists a $k$-GDD of type $\{g(1,1), g(1,2), \ldots, g(1,s_1), g(2,1), 
g(2,2), \ldots, g(2,s_2),$ $\ldots, g(m,1), g(m,2), \ldots, g(m,s_m)\}$.
\end{construction}

Another tool for constructing GDDs from other GDDs is {\it Wilson's fundamental GDD construction}~\cite{Wilson}.
A version of this construction is given in Construction~\ref{construction.Wilsons} in the form in which it will be used in Lemma~\ref{lemma.4gdd.410283.and.413283}.
This construction is valid by~\cite[Theorem 1.4.9]{fmy}; see also the proof of that theorem.

\begin{construction} \label{construction.Wilsons}
Suppose there exists a $k$-GDD of type~$g_1^{u_1} g_2^{u_2} \cdots g_m^{u_m}$ and there exists a $\operatorname{TD}(k,h)$. 
Then there exists a $k$-GDD of type~$(h g_1)^{u_1} (h g_2)^{u_2} \cdots (hg_m)^{u_m}$.
\end{construction}

In Lemma~\ref{lemma.necessary.conditions.4t7s}, we give the necessary conditions for the existence of a $4$-GDD of type~$4^t 7^s$.
However, we first need to consider the result in Lemma~\ref{lemma.4gdd.g1tg2s.tgeq4.or.sgeq4.or.t=s=3} which simplifies some of the necessary conditions in Theorem~\ref{necessary}, specifically when the $4$-GDD has up to two group sizes.

\begin{lemma} \label{lemma.4gdd.g1tg2s.tgeq4.or.sgeq4.or.t=s=3}
Suppose there exists a $4$-GDD of type $g_1^t g_2^s$. Then $t \geq 4$ or $s \geq 4$ or $t = s = 3$.

\begin{proof}
By Condition (1) in Theorem~\ref{necessary}, $t+s \geq 4$. 
If $t+s = 4$, then by Condition (5), either $t = 4$ or $s = 4$.
If $t+s = 5$, then by Condition (6), either $t = 4$ or $s = 4$.
If $t+s = 6$ and neither $t \geq 4$ nor $s \geq 4$, then $t = s = 3$.
If $t + s \geq 7$, then either $t \geq 4$ or $s \geq 4$.
\end{proof}
\end{lemma}

\begin{lemma} 
\label{lemma.necessary.conditions.4t7s}
If a $4$-GDD of type $4^t 7^s$ exists then all of the following conditions hold$:$
\begin{itemize}
    \item $t + s \equiv 1 \pmod{3};$
    \item $s \equiv 0$ or $1 \pmod{4};$ and
    \item either $t \geq 4$ or $s \geq 4$.
\end{itemize}

\begin{proof}
Firstly, note that a $4$-GDD of type~$4^t 7^s$ has $t+s$ groups.
Since $4 \not\equiv 0 \pmod{3}$ and $7 \not\equiv 0 \pmod{3}$, it follows from
Lemma~\ref{lemma.4gdd.necessary.condition.alternative.group.sizes.mod3} that $t+s \equiv 1 \pmod{3}$.

Next, observe that the only groups of odd size are the groups of size~$7$.
Hence, by Lemma~\ref{lemma.4gdd.number.of.odd.sized.groups}, it follows that $s \equiv 0$ or $1 \pmod{4}$.

Finally, consider when $t = s = 3$. Then $t + s = 6$ which contradicts $t+s \equiv 1 \pmod{3}$.
Hence, by Lemma~\ref{lemma.necessary.conditions.4t7s}, it follows that either $t \geq 4$ or $s \geq 4$.
\end{proof}
\end{lemma}

\begin{lemma}
\label{lemma.4gdd.4t7s.v4or7mod12}
Suppose that $v = 4t+7s$.
Then the conditions that $t + s \equiv 1 \pmod{3}$ and $s \equiv 0 \pmod{4}$ are equivalent to the condition that $v \equiv 4 \pmod{12}$.
Similarly, the conditions that $t + s \equiv 1 \pmod{3}$ and $s \equiv 1 \pmod{4}$ are equivalent to the condition that $v \equiv 7 \pmod{12}$.

\begin{proof}
Firstly, note that $t + s \equiv 1 \pmod{3}$ implies that $4t + 4s \equiv 1 \pmod{3}$, so $v \equiv 4t + 7s \equiv 1 \pmod{3}$.
Also, observe that $s \equiv 0$ or $1 \pmod{4}$ implies that $3s \equiv 0$ or $3 \pmod{4}$, so $v \equiv 4t + 7s \equiv 0$ or $3 \pmod{4}$.
If $v \equiv 1 \pmod{3}$ and $v \equiv 0 \pmod{4}$, then $v \equiv 4 \pmod{12}$.
If $v \equiv 1 \pmod{3}$ and $v \equiv 3 \pmod{4}$, then $v \equiv 7 \pmod{12}$.

Now, note that $v \equiv 4t+7s \equiv t+s \pmod{3}$ and that $v \equiv 4t+7s \equiv -s \pmod{4}$.
If $v \equiv 4 \pmod{12}$ then $v \equiv 1 \pmod{3}$ and $v \equiv 0 \pmod{4}$.
This means that $t+s \equiv 1 \pmod{3}$ and $-s \equiv 0 \pmod{4}$, so $s \equiv 0 \pmod{4}$.
If $v \equiv 7 \pmod{12}$ then $v \equiv 1 \pmod{3}$ and $v \equiv 3 \pmod{4}$.
This means that $t+s \equiv 1 \pmod{3}$ and $-s \equiv 3 \pmod{4}$, so $s \equiv 1 \pmod{4}$.
\end{proof}
\end{lemma}

We can now begin to construct $4$-GDDs of type~$4^t 7^s$.
Following from the result in Lemma~\ref{lemma.4gdd.4t7s.v4or7mod12}, we find it convenient to organise the constructions by the number of points $v = 4t+7s$.
Specifically, we organise the constructions by the value of~$v$ modulo~$84$ where $v \equiv 4$ or $7 \pmod{12}$.

Note that throughout this paper, the parameters~$t$ and~$s$ are always non-negative integers.

Firstly, in Lemmas~\ref{lemma.4gdd.410283.and.413283} and~\ref{lemma.4gdd.4^16_7^9}, we establish the existence of some $4$-GDDs that will be used for some of the constructions in Lemma~\ref{lemma.4t7s.vleq151}.

\begin{lemma}
\label{lemma.4gdd.410283.and.413283}
There exist $4$-GDDs of types~$4^{10} 28^3$ and~$4^{13} 28^3$.

\begin{proof}
For type~$4^{10} 28^3$, start with a $4$-GDD of type~$1^{10} 7^3$ which can be obtained by forming a block on each group of size~$4$ in a $4$-GDD of type $1^2 4^2 7^3$ which exists by \cite[Table~8]{ABC.50less}.
For type~$4^{13} 28^3$, start with a $4$-GDD of type~$1^{13} 7^3$ which can be obtained by forming a block on each group of size $4$ in a $4$-GDD of type~$1^1 4^3 7^3$ which exists by \cite[Table~3]{ABBCF.1mod3}.
Then, in both cases, apply Construction~\ref{construction.Wilsons} with $h=4$.
\end{proof}
\end{lemma}

\begin{lemma}
\label{lemma.4gdd.4^16_7^9}
There exist $4$-GDDs of types~$1^{64} 7^9$ and~$4^{16} 7^9$.

\begin{proof}
A $4$-GDD of type~$1^{64} 7^9$ is given in Table~\ref{16479}.
We use this $4$-GDD to construct a $4$-GDD of type~$4^{16} 7^9$.
This $4$-GDD has $16$~disjoint blocks that do not contain any points from the groups of size~$7$.
Nine of these blocks are obtained by adding $0,1,2, \ldots, 8$ to the base block $\{p_{0},p_{11},q_{0},r_{0}\}$. 
Three of these blocks are obtained by adding $0,1,2$ to the base block $\{q_{10},q_{16},r_{11},r_{18}\}$. 
The remaining four blocks are $\{p_{20},q_{8},r_{14},\infty\}$, $\{p_{19},q_{9},q_{20},r_{16}\}$, 
  $\{p_{10},r_{9},r_{15},r_{17}\}$ and  $\{q_{14},q_{15},q_{19},r_{8}\}$.
Removing these $16$~blocks from the $4$-GDD of type~$1^{64} 7^9$ gives a $4$-GDD of type~$4^{16} 7^9$.
\end{proof}
\end{lemma}

\begin{lemma} 
\label{lemma.4t7s.vleq151}
Suppose that $v = 4t+7s$ where $v \equiv 4$ or $7 \pmod{12}$ and either $t \geq 4$ or $s \geq 4$. 
Suppose also that $v \leq 151$.
Then there exists a $4$-GDD of type~$4^t 7^s$ except possibly when
\begin{itemize}
    \item $v =  76:$ $(t,s) = (12,4)$ or $(5,8);$
    \item $v =  79:$ $(t,s) = (11,5);$
    \item $v = 100:$ $(t,s) = (11,8);$
    \item $v = 103:$ $(t,s) = (17,5),$ $(10,9)$ or $(3,13);$
    \item $v = 115:$ $(t,s) = (20,5);$
    \item $v = 124:$ $(t,s) = (3,16);$
    \item $v = 127:$ $(t,s) = (9,13)$ or $(2,17);$
    \item $v = 139:$ $(t,s) = (19,9),$ $(12,13)$ or $(5,17).$
\end{itemize}

\begin{proof}
Constructions of these $4$-GDDs are given in Table~\ref{table.4t7s.vleq76} for $v \leq 76$ and in Table~\ref{table.4t7s.v79-151} for $79 \leq v \leq 151$.

Some of the $4$-GDDs in Table~\ref{table.4t7s.v79-151} are constructed using Construction~\ref{construction.fillin}.
Start with the given input $4$-GDD which exists by the reference given in the last column of the table.
Then, apply Construction~\ref{construction.fillin} using the listed value of~$u$.
The required fill-in $4$-GDDs are also given and references for these are given in Table~\ref{table.4t7s.vleq76}.

The $4$-GDDs in Table~\ref{table.4t7s.vleq76} and the remaining $4$-GDDs in Table~\ref{table.4t7s.v79-151} are constructed directly in this paper or are constructed elsewhere.
A reference is given for each of these $4$-GDDs.
The direct constructions in this paper can be found in the Appendix.
As explained in Section~\ref{s:some.direct.constructions}, these were found by assuming the existence of a cyclic automorphism group, $\mathbb{G}$, assigning a selection of base blocks, and then developing these base blocks to produce the blocks of the required design.
\end{proof}
\end{lemma}

\begin{table}[ht!]
\caption{Constructions for $4$-GDDs of type $4^t 7^s$ with $v \leq 76$ points in Lemma~\ref{lemma.4t7s.vleq151}.}
\label{table.4t7s.vleq76}
{\noindent
\begin{center}
  \begin{tabular}{|c|l|l|c|l|l|} \hline
    $v$ & Types & Reference & $v$ & Types & Reference \\ \hline
    $16$ & $4^4$     & \cite[Lemma~6.13]{BSH}              & $55$ & $4^{12} 7^1$ & \cite[Theorem~3.18]{GeLing.2004.gum1}  \\ \hline
    $28$ & $4^7$     & \cite[Lemma~6.13]{BSH}              &      & $4^5 7^5$    & Table~\ref{4575}  \\ \hline
         & $7^4$     & \cite[Lemma~3.5]{Hanani.1975}       & $64$ & $4^{16}$     & \cite[Lemma~6.13]{BSH}  \\ \hline
    $31$ & $4^6 7^1$ & \cite[Lemma~3.17]{GeLing.2004.gum1} &      & $4^9 7^4$    & Table~\ref{4974}  \\ \hline
    $40$ & $4^{10}$  & \cite[Lemma~6.13]{BSH}              &      & $4^2 7^8$    & Table~\ref{4278}  \\ \hline
         & $4^3 7^4$ & \cite[Table~8]{ABBCF.1mod3}         & $67$ & $4^{15} 7^1$ & \cite[Theorem~3.18]{GeLing.2004.gum1}  \\ \hline
    $43$ & $4^9 7^1$ & \cite[Lemma~3.17]{GeLing.2004.gum1} &      & $4^8 7^5$    & Table~\ref{4875}  \\ \hline
         & $4^2 7^5$ & \cite[Table~15]{ABBCF.1mod3}        &      & $4^1 7^9$    & \cite[Theorem~2.11]{gezhu}  \\ \hline
    $52$ & $4^{13}$  & \cite[Lemma~6.13]{BSH}              & $76$ & $4^{19}$     & \cite[Lemma~6.13]{BSH} \\ \hline
         & $4^6 7^4$ & Table~\ref{4674}                    &      & $4^{12} 7^4$ & Unknown \\ \hline
         &           &                                     &      & $4^5 7^8$    & Unknown \\ \hline
  \end{tabular}
\end{center}}
\end{table}

\begin{table}[ht!]
\caption{Constructions for $4$-GDDs of type $4^t 7^s$ with $79 \leq v \leq 151$ points in Lemma~\ref{lemma.4t7s.vleq151}.}
\label{table.4t7s.v79-151}
{\noindent
\begin{center}
  \begin{tabular}{|c|l|l|c|l|l|} \hline
    $v$ & Types & Input $4$-GDD & $u$ & Fill-in $4$-GDD & Reference \\ \hline

     $79$ & $4^{18} 7^1$ &  &  &  & Lemma~\ref{lemma.4pn1} \\ \hline
          & $4^{11} 7^5$ &  &  &  & Unknown \\ \hline
          & $4^4 7^9$ & $7^9 16^1$ & $0$ & $4^4$ & Lemma~\ref{lemma.7pn1} \\ \hline
          
     $88$ & $4^{22}$, $4^{15} 7^4$ & $4^{15} 28^1$ & $0$ & $4^7, 7^4$ & Lemma~\ref{lemma.4pn1} \\ \hline
          & $4^8 7^8$ &  &  &  & Table~\ref{4878} \\ \hline
          & $4^1 7^{12}$ &  &  &  & Lemma~\ref{lemma.7pn1} \\ \hline

     $91$ & $4^{21} 7^1$ &  &  &  & Lemma~\ref{lemma.4pn1} \\ \hline
          & $4^{14} 7^5$ &  &  &  & Table~\ref{41475} \\ \hline
          & $4^7 7^9$, $7^{13}$ & $7^9 28^1$ & $0$ & $4^7, 7^4$ & Lemma~\ref{lemma.7pn1} \\ \hline
          
    $100$ & $4^{25}$, $4^{18} 7^4$ & $4^{18} 28^1$ & $0$ & $4^7, 7^4$ & Lemma~\ref{lemma.4pn1} \\ \hline
          & $4^{11} 7^8$ &  &  &  & Unknown \\ \hline
          & $4^4 7^{12}$ & $7^{12} 16^1$ & $0$ & $4^4$ & Lemma~\ref{lemma.7pn1} \\ \hline

    $103$ & $4^{24} 7^1$ &  &  &  & Lemma~\ref{lemma.4pn1} \\ \hline
          & $4^{17} 7^5$, $4^{10} 7^9$, $4^3 7^{13}$ &  &  &  & Unknown \\ \hline
          
    $112$ & $4^{28}$, $4^{21} 7^4$, $4^{14} 7^8$, $4^7 7^{12}$, $7^{16}$ & $28^4$ & $0$ & $4^7, 7^4$ & Lemma~\ref{lemma.gp} \\ \hline

    $115$ & $4^{27} 7^1$ &  &  &  & Lemma~\ref{lemma.4pn1} \\ \hline
          & $4^{20} 7^5$ &  &  &  & Unknown \\ \hline
          & $4^{13} 7^9$, $4^6 7^{13}$ & $21^4 24^1$ & $7$ & $4^7$, $7^4$, $4^6 7^1$ & Lemma~\ref{lemma.gpn1.g3mod6} \\ \hline
          
    $124$ & $4^{31}$, $4^{24} 7^4$, $4^{17} 7^8$, $4^{10} 7^{12}$ & $4^{10} 28^3$ & $0$ & $4^7$, $7^4$ & Lemma~\ref{lemma.4gdd.410283.and.413283} \\ \hline
          & $4^3 7^{16}$ &  &  &  & Unknown \\ \hline

    $127$ & $4^{30} 7^1$, $4^{23} 7^5$ & $24^4 27^1$ & $4$ & $4^7$, $7^4$, $4^6 7^1$ & Lemma~\ref{lemma.gpn1.0mod6} \\ \hline
          & $4^{16} 7^9$ &  &  &  & Lemma~\ref{lemma.4gdd.4^16_7^9} \\ \hline
          & $4^9 7^{13}$, $4^2 7^{17}$ &  &  &  & Unknown \\ \hline

    $136$ & $4^{34}$, $4^{27} 7^4$, $4^{20} 7^8$, $4^{13} 7^{12}$ & $4^{13} 28^3$ & $0$ & $4^7$, $7^4$ & Lemma~\ref{lemma.4gdd.410283.and.413283} \\ \hline
          & $4^6 7^{16}$ & $7^{15} 31^1$ & $0$ & $4^6 7^1$ & Lemma~\ref{lemma.7pn1} \\ \hline
          
    $139$ & $4^{33} 7^1$, $4^{26} 7^5$ & $24^4 36^1$ & 7 & $4^6 7^1$, $4^9 7^1$, $4^2 7^5$ & Lemma~\ref{lemma.gpn1.0mod6} \\ \hline
          & $4^{19} 7^9$, $4^{12} 7^{13}$, $4^5 7^{17}$ &  &  &  & Unknown \\ \hline

    $148$ & $4^{37}$, $4^{30} 7^4$, $4^{23} 7^8$, $4^{16} 7^{12}$, $4^9 7^{16}$ & $36^4$ & $4$ & $4^{10}, 4^3 7^4$ & Lemma~\ref{lemma.gp} \\ \hline
          & $4^2 7^{20}$ & $7^{15} 43^1$ & $0$ & $4^2 7^5$ & Lemma~\ref{lemma.7pn1} \\ \hline
          
    $151$ & $4^{36} 7^1$, $4^{29} 7^5$, $4^{22} 7^9$, $4^{15} 7^{13}$, $4^8 7^{17}$ & $36^4$ & $7$ & $4^9 7^1, 4^2 7^5$ & Lemma~\ref{lemma.gp} \\ \hline
          & $4^1 7^{21}$ &  &  &  & Lemma~\ref{lemma.7pn1} \\ \hline

  \end{tabular}
\end{center}}
\end{table}

We use the $4$-GDDs in Lemma~\ref{lemma.4t7s.vleq151} to construct infinite families of the remaining $4$-GDDs of type $4^t 7^s$.
Note that $v \equiv 4$ or $7 \pmod{12}$ implies that $v \equiv 4, 7, 16, 19, \ldots, 76$ or $79 \pmod{84}$.
We first consider the cases when $v \not\equiv 19, 76$ or $79 \pmod{84}$ in Lemma~\ref{lemma.4t7s.inf.family.main}.
The cases when $v \equiv 19, 76$ or $79 \pmod{84}$ are then handled in Lemmas~\ref{lemma.4t7s.adhoc.rec.v19mod84} to \ref{lemma.4t7s.inf.family.v79mod84}.

\begin{lemma} 
\label{lemma.4t7s.inf.family.main}
Suppose that $v = 4t+7s$ where $v \equiv 4$ or $7 \pmod{12}$.
Suppose also that $v \geq 172$ where $v \not\equiv 19, 76$ or $79 \pmod{84}$.
Then there exists a $4$-GDD of type~$4^t 7^s$.
\end{lemma}

\begin{proof}
If $t = 0$ or $s \in \{0,1\}$, 
then existence is given by Lemmas~\ref{lemma.4gdd.g^p} and \ref{lemma.4pn1}, 
so we may assume that $t > 0$ and $s \geq 2$.

Suppose that $\ell \equiv 4t+7s \pmod{84}$ where $0 \leq \ell < 84$ and $\ell \neq 19, 76$ or $79$.
Recall that $v \not\equiv 19, 76$ or $79 \pmod{84}$, so it follows that $0 \leq \ell \leq 67$.
Then set $m = (4t+7s-\ell)/28$.
Now, note that $m \equiv 0 \pmod{3}$, so $v \geq 172$ implies that $m \geq 6$.
This means that $\ell \leq 67 \leq 14(m-1)$.
Hence, by Lemma~\ref{lemma.4gdd.28p_n1}, there exists a $4$-GDD of type $28^m \ell^1$.

Start with a 4-GDD of type~$28^m \ell^1$.
Then, apply Construction~\ref{construction.fillin} with $u=0$.  

If $\ell = 4$, 
then set $x = s/4$.

If $\ell = 7$, 
then set $x = (s-1)/4$.

If $\ell = 16$, 
then fill in the group of size $\ell$ with a $4$-GDD of type $4^4$
and set $x = s/4$.

If $\ell = 31$, 
then fill in the group of size $\ell$ with a $4$-GDD of type $4^6 7^1$ and set $x = (s-1)/4$.

If $\ell \in \{28, 40, 52\}$, then there are at least four groups of size $7$.
Fill in the group of size $\ell$ with a $4$-GDD of type $7^4$, $4^3 7^4$ or $4^6 7^4$ for $\ell=28, 40, 52$, respectively,
and set $x = (s-4)/4$.

If $\ell \in \{43, 55\}$, then there are at least five groups of size $7$.
Fill in the group of size $\ell$ with a $4$-GDD of type $4^2 7^5$ or $4^5 7^5$ for $\ell=43, 55$, respectively,
and set $x = (s-5)/4$.

If $\ell = 64$, then there are at least four groups of size $7$.
If $s = 4$, 
then fill in the group of size $\ell$ with a $4$-GDD of type $4^9 7^4$ and set $x = 0$. 
If $s \geq 8$, 
then fill in the group of size $\ell$ with a $4$-GDD of type $4^2 7^8$
and set $x = (s-8)/4$.

If $\ell = 67$, then there are at least five groups of size $7$.
If $s = 5$, 
then fill in the group of size $\ell$ with a $4$-GDD of type $4^8 7^5$ 
and set $x = 0$. 
If $s \geq 9$, 
then fill in the group of size $\ell$ with a $4$-GDD of type $4^1 7^9$
and set $x = (s-9)/4$.

Finally, fill in $x$ groups of size $28$ with a $4$-GDD of type $7^4$ and the remaining $m-x$ groups of size $28$, if any, with a 4-GDD of type $4^7$.
\end{proof}


\begin{lemma} \label{lemma.4t7s.adhoc.rec.v19mod84}
Suppose that $v = 4t+7s$ where $v \equiv 19 \pmod{84}$ and $187 \leq v \leq 523$.
Then there exists a $4$-GDD of type~$4^t 7^s$.

\begin{proof}
Constructions for these $4$-GDDs are given in Table~\ref{table.4t7s.adhoc.rec.v19mod84}. 
For each pair $(t,s)$, we start with a given input $4$-GDD and then apply Construction~\ref{construction.fillin} using the given value of~$u$. 
We also give the required fill-in designs for each construction. 
Some input $4$-GDDs have a type of the form~$g^p n^1$ with $g \in \{7, 36, 39\}$ in which case they exist by Lemma~\ref{gpn1.0mod6}, \ref{lemma.gpn1.g3mod6} or \ref{lemma.7pn1}.
All other input $4$-GDDs have a type of the form~$g^p$ and exist by Lemma~\ref{lemma.4gdd.g^p}.
\end{proof}
\end{lemma}

\begin{table}[hbt!]\small
\caption{Constructions for $4$-GDDs of type~$4^t 7^s$ with $187 \leq 4t+7s \leq 523$ in Lemma~\ref{lemma.4t7s.adhoc.rec.v19mod84}.}
\label{table.4t7s.adhoc.rec.v19mod84}
{\noindent
\begin{center}
  \begin{tabular}{|c|l|l|c|l|} \hline
    $v$ & Types & Input $4$-GDD & $u$ & Fill-in 4-GDDs \\ \hline
        
    187 & $4^{45} 7^1, 4^{38} 7^5, 4^{31} 7^9, 4^{24} 7^{13}, 4^{17} 7^{17}$ & $36^5$ & $7$ & $4^9 7^1, 4^2 7^5$  \\ \hline
        & $4^{10} 7^{21}, 4^3 7^{25}$ & $7^{21} 40^1$ & $0$ & $4^{10}, 4^3 7^4$ \\ \hline
        
    271 & $4^{66} 7^1, 4^{59} 7^5, 4^{52} 7^9, \ldots, 4^{17} 7^{29}$ & $36^6 51^1$ & $4$ & $4^{10}, 4^3 7^4, 4^{12} 7^1, 4^5 7^5$ \\ \hline
        & $4^{10} 7^{33}, 4^3 7^{37}$ & $7^{33} 40^1$ & $0$ & $4^{10}, 4^3 7^4$ \\ \hline
        
    355 & $4^{87} 7^1, 4^{80} 7^5, 4^{73} 7^9, \ldots, 4^{17} 7^{41}$ & $36^8 63^1$ & $4$ & $4^{10}, 4^3 7^4, 4^{15} 7^1, 4^1 7^9$ \\ \hline
        & $4^{10} 7^{45}, 4^3 7^{49}$ & $7^{45} 40^1$ & $0$ & $4^{10}, 4^3 7^4$ \\ \hline

    439 & $4^{108} 7^1, 4^{101} 7^5, 4^{94} 7^9, \ldots, 4^{24} 7^{49}$ & $36^{10} 75^1$ & $4$ & $4^{10}, 4^3 7^4, 4^{18} 7^1, 4^4 7^9$ \\ \hline
        & $4^{17} 7^{53}$ & $39^9 84^1$ & $4$ & $4^2 7^5, 4^8 7^8$ \\ \hline
        & $4^{10} 7^{57}, 4^3 7^{61}$ & $7^{57} 40^1$ & $0$ & $4^{10}, 4^3 7^4$ \\ \hline

    523 & $4^{129} 7^1, 4^{122} 7^5, 4^{115} 7^9, \ldots, 4^{24} 7^{61}$ & $36^{12} 87^1$ & $4$ & $4^{10}, 4^3 7^4, 4^{21} 7^1, 7^{13}$ \\ \hline
        & $4^{17} 7^{65}$ & $39^{12} 51^1$ & $4$ & $4^2 7^5, 4^5 7^5$ \\ \hline
        & $4^{10} 7^{69}, 4^3 7^{73}$ & $7^{69} 40^1$ & $0$ & $4^{10}, 4^3 7^4$ \\ \hline
  \end{tabular}
\end{center}}
\end{table}

\begin{lemma} 
\label{lemma.4t7s.inf.family.v19mod84}
Suppose that $v = 4t+7s$ where $v \equiv 19 \pmod{84}$ and $v \geq 607$.
Then there exists a $4$-GDD of type~$4^t 7^s$.

\begin{proof}
If $t = 1$ or $s = 1$, 
then existence is given by Lemmas~\ref{lemma.4pn1} and \ref{lemma.7pn1}.
Since $v \equiv 19 \pmod{84}$, it follows that $v \equiv 7 \pmod{12}$.
Thus, by Lemma~\ref{lemma.4gdd.4t7s.v4or7mod12}, we have $s \equiv 1 \pmod{4}$
so we may assume that $t > 1$ and $s \geq 5$.

Set $m = (4t+7s-187)/28$.
Note that $v \geq 607$ implies that $m \geq 15$.
This means that $187 \leq 14(m-1)$.
Hence, by Lemma~\ref{lemma.4gdd.28p_n1}, there exists a $4$-GDD of type~$28^m 187^1$.

Start with a $4$-GDD of type~$28^m 187^1$.
Then, apply Construction~\ref{construction.fillin} with $u=0$.

If $s \leq 21$, then fill in the group of size $187$ with a $4$-GDD of type~$4^{(187-7s)/4} 7^s$ which exists by Lemma~\ref{lemma.4t7s.adhoc.rec.v19mod84}. 
Then, fill in all $m$ groups of size $28$ with a $4$-GDD of type~$4^7$.

If $s \geq 25$, then fill in the group of size $187$ with a $4$-GDD of type~$4^3 7^{25}$ which exists by Lemma~\ref{lemma.4t7s.adhoc.rec.v19mod84}. 
Then, fill in $(s-25)/4$ groups of size~$28$ with a $4$-GDD of type~$7^4$ 
and fill in the remaining $m-(s-25)/4 = (t-3)/7$ groups of size~$28$ with a $4$-GDD of type~$4^7$.
\end{proof}
\end{lemma}

\begin{lemma} 
\label{lemma.4t7s.adhoc.rec.v76mod84}
Suppose that $v = 4t+7s$ where $v \equiv 76 \pmod{84}$ and $160 \leq v \leq 496$.
Then there exists a $4$-GDD of type~$4^t 7^s$.

\begin{proof}
Constructions for these $4$-GDDs are given in Table~\ref{table.4t7s.adhoc.rec.v76mod84}. 
For each pair $(t,s)$, we start with a given input $4$-GDD and then apply Construction~\ref{construction.fillin} using the given value of~$u$. 
We also give the required fill-in designs for each construction.
Some input $4$-GDDs have a type of the form~$g^p n^1$ with $g \in \{7, 36\}$ in which case they exist by Lemma~\ref{gpn1.0mod6} or \ref{lemma.7pn1}.
All other input $4$-GDDs have a type of the form~$g^p$ and exist by Lemma~\ref{lemma.4gdd.g^p}.
\end{proof}
\end{lemma}

\begin{table}[hbt!]\small
\caption{Constructions for $4$-GDDs of type~$4^t 7^s$ with $160 \leq v \leq 496$ in Lemma~\ref{lemma.4t7s.adhoc.rec.v76mod84}.}
\label{table.4t7s.adhoc.rec.v76mod84}
{\noindent
\begin{center}
  \begin{tabular}{|c|l|l|c|l|} \hline
    $v$ & Types & Input $4$-GDD type & $u$ & Fill-in 4-GDD types \\ \hline
        
    160 & $4^{40}, 4^{33} 7^4, 4^{26} 7^8, 4^{19} 7^{12}, 4^{12} 7^{16}$ & $40^4$ & $0$ & $4^{10}, 4^3 7^4$  \\ \hline
        & $4^5 7^{20}$ & $39^4$ & $4$ & $4^2 7^5$ \\ \hline
        
    244 & $4^{61}, 4^{54} 7^4, 4^{47} 7^8, \ldots, 4^{19} 7^{24}$ & $36^6 24^1$ & $4$ & $4^{10}, 4^3 7^4, 4^7$  \\ \hline
        & $4^{12} 7^{28}, 4^5 7^{32}$ & $7^{27} 55^1$ & $0$ & $4^{12} 7^1, 4^5 7^5$ \\ \hline
        
    328 & $4^{82}, 4^{75} 7^4, 4^{68} 7^8, \ldots, 4^{19} 7^{36}$ & $36^9$ & $4$ & $4^{10}, 4^3 7^4$  \\ \hline
        & $4^{12} 7^{40}, 4^5 7^{44}$ & $7^{39} 55^1$ & $0$ & $4^{12} 7^1, 4^5 7^5$ \\ \hline

    412 & $4^{103}, 4^{96} 7^4, 4^{89} 7^8, \ldots, 4^{19} 7^{48}$ & $36^9 84^1$ & $4$ & $4^{10}, 4^3 7^4, 4^{22}, 4^1 7^{12}$  \\ \hline
        & $4^{12} 7^{52}, 4^5 7^{56}$ & $7^{51} 55^1$ & $0$ & $4^{12} 7^1, 4^5 7^5$ \\ \hline

    496 & $4^{124}, 4^{117} 7^4, 4^{110} 7^8, \ldots, 4^{26} 7^{56}$ & $36^{13} 24^1$ & $4$ & $4^{10}, 4^3 7^4, 4^7, 7^4$  \\ \hline
        & $4^{19} 7^{60}$ & $7^{60} 76^1$ & $0$ & $4^{19}$ \\ \hline
        & $4^{12} 7^{64}, 4^5 7^{68}$ & $7^{63} 55^1$ & $0$ & $4^{12} 7^1, 4^5 7^5$ \\ \hline
  \end{tabular}
\end{center}}
\end{table}

\begin{lemma} 
\label{lemma.4t7s.inf.family.v76mod84}
Suppose that $v = 4t+7s$ where $v \equiv 76 \pmod{84}$ and $v \geq 580$.
Then there exists a $4$-GDD of type~$4^t 7^s$.

\begin{proof}
If $t = 0$ or $s = 0$, 
then existence is given by Lemma~\ref{lemma.4gdd.g^p}.
Since $v \equiv 76 \pmod{84}$, it follows that $v \equiv 4 \pmod{12}$.
Thus, by Lemma~\ref{lemma.4gdd.4t7s.v4or7mod12}, we have $s \equiv 0 \pmod{4}$
so we may assume that $t > 0$ and $s \geq 4$.

Set $m = (4t+7s-160)/28$.
Note that $v \geq 580$ implies that $m \geq 15$.
This means that $160 \leq 14(m-1)$.
Hence, by Lemma~\ref{lemma.4gdd.28p_n1}, there exists a $4$-GDD of type $28^m 160^1$.

Start with a $4$-GDD of type~$28^m 160^1$.
Then, apply Construction~\ref{construction.fillin} with $u=0$.

If $s \leq 16$, then fill in the group of size~$160$ with a $4$-GDD of type~$4^{(160-7s)/4} 7^s$ which exists by Lemma~\ref{lemma.4t7s.adhoc.rec.v76mod84}. 
Then, fill in all $m$~groups of size~$28$ with a $4$-GDD of type~$4^7$.

If $s \geq 20$, then fill in the group of size~$160$ with a $4$-GDD of type~$4^5 7^{20}$ which exists by Lemma~\ref{lemma.4t7s.adhoc.rec.v76mod84}. 
Then, fill in $(s-20)/4$ groups of size~$28$ with a $4$-GDD of type~$7^4$ 
and fill in the remaining $m-(s-20)/4 = (t-5)/7$ groups of size~$28$ with a $4$-GDD of type~$4^7$.
\end{proof}
\end{lemma}

\begin{lemma} 
\label{lemma.4t7s.adhoc.rec.v79mod84}
Suppose that $v = 4t+7s$ where $v \equiv 79 \pmod{84}$ and $v \in \{163,247\}$.
Then there exists a $4$-GDD of type~$4^t 7^s$.

\begin{proof}
Constructions for these $4$-GDDs are given in Table~\ref{table.4t7s.adhoc.rec.v79mod84}. 
For each pair $(t,s)$, we start with a given input $4$-GDD and then apply Construction~\ref{construction.fillin} using the given value of~$u$. 
We also give the required fill-in designs for each construction. 
All input $4$-GDDs have a type of the form~$g^p n^1$ with $g \in \{7, 12, 36, 39\}$ in which case they exist by Lemma~\ref{gpn1.0mod6}, \ref{lemma.gpn1.g3mod6} or~\ref{lemma.7pn1}.
\end{proof}
\end{lemma}

\begin{table}[hbt!]\small
\caption{Constructions for $4$-GDDs of type~$4^t 7^s$ with $v \in\{163,247\}$ in Lemma~\ref{lemma.4t7s.adhoc.rec.v79mod84}.}
\label{table.4t7s.adhoc.rec.v79mod84}
{\noindent
\begin{center}
  \begin{tabular}{|c|l|l|c|l|} \hline
    $v$ & Types & Input $4$-GDD type & $u$ & Fill-in $4$-GDD types \\ \hline
        
    163 & $4^{39} 7^1$ & $12^{13} 3^1$ & $4$ & $4^4$  \\ \hline
        & $4^{32} 7^5, 4^{25} 7^9, 4^{18} 7^{13}, 4^{11} 7^{17}, 4^4 7^{21}$ & $39^4 3^1$ & $4$ & $4^2 7^5, 4^9 7^1$ \\ \hline

    247 & $4^{60} 7^1, 4^{53} 7^5, 4^{46} 7^9, \ldots, 4^{11} 7^{29}$ & $36^5 63^1$ & $4$ & $4^{10}, 4^3 7^4, 4^{15} 7^1, 4^1 7^9$  \\ \hline
        & $4^4 7^{33}$ & $7^{33} 16^1$ & $0$ & $4^4$ \\ \hline
  \end{tabular}
\end{center}}
\end{table}

\begin{lemma} 
\label{lemma.4t7s.inf.family.v79mod84}
Suppose that $v = 4t+7s$ where $v \equiv 79 \pmod{84}$ and $v \geq 331$.
Then there exists a $4$-GDD of type~$4^t 7^s$.

\begin{proof}
If $t = 0$ or $s = 1$, 
then existence is given by Lemma~\ref{lemma.4gdd.g^p} or~\ref{lemma.4pn1}.
Since $v \equiv 79 \pmod{84}$, it follows that $v \equiv 7 \pmod{12}$.
Thus, by Lemma~\ref{lemma.4gdd.4t7s.v4or7mod12}, we have $s \equiv 1 \pmod{4}$
so we may assume that $t > 0$ and $s \geq 5$.

Set $m = (4t+7s-79)/28$.
Note that $v \geq 331$ implies that $m \geq 9$.
This means that $79 \leq 14(m-1)$.
Hence, by Lemma~\ref{lemma.4gdd.28p_n1}, there exists a $4$-GDD of type~$28^m 79^1$.

Start with a $4$-GDD of type~$28^m 79^1$. 
Then, apply Construction~\ref{construction.fillin} with $u=0$.

If $s = 5$, then fill in the group of size~$79$ with a $4$-GDD of type~$4^{18} 7^1$ which exists by Lemma~\ref{lemma.4pn1}. 
Then, fill in one group of size~$28$ with a $4$-GDD of type~$7^4$ and the remaining $m-1$ groups of size~$28$ with a $4$-GDD of type~$4^7$.

If $s \geq 9$, then fill in the group of size~$79$ with a $4$-GDD of type~$4^4 7^9$ which exists by Lemma~\ref{lemma.4t7s.vleq151}. 
Then, fill in $(s-9)/4$ groups of size~$28$ with a $4$-GDD of type~$7^4$ 
and fill in the remaining $m-(s-9)/4 = (t-4)/7$ groups of size~$28$ with a $4$-GDD of type~$4^7$.
\end{proof}
\end{lemma}

\section{Summary}
\label{s:summary}

\begin{theorem}
\label{theorem.4t7s.summary}
Suppose that $t + s \equiv 1 \pmod{3}$ and $s \equiv 0$ or $1 \pmod{4}$ where either $t \geq 4$ or $s \geq 4$. 
Then there exists a $4$-GDD of type~$4^t 7^s$ except possibly when
\begin{itemize}
    \item $v = \phantom{1}76:$ $4^{12} 7^4$ or $4^5 7^8;$
    \item $v = \phantom{1}79:$ $4^{11} 7^5;$
    \item $v = 100:$ $4^{11} 7^8;$
    \item $v = 103:$ $4^{17} 7^5,$ $4^{10} 7^9$ or $4^3 7^{13};$
    \item $v = 115:$ $4^{20} 7^5;$
    \item $v = 124:$ $4^3 7^{16};$
    \item $v = 127:$ $4^9 7^{13}$ or $4^2 7^{17};$
    \item $v = 139:$ $4^{19} 7^9,$ $4^{12} 7^{13}$ or $4^5 7^{17}.$
\end{itemize}
\begin{proof}
Recall from Lemma~\ref{lemma.4gdd.4t7s.v4or7mod12} that the conditions that $t + s \equiv 1 \pmod{3}$ and $s \equiv 0$ or $1 \pmod{4}$ are equivalent to the condition that $v \equiv 4$ or $7 \pmod{12}$ where $v = 4t+7s$.
If $v \leq 151$ where either $t \geq 4$ or $s \geq 4$, then the $4$-GDDs exists by Lemma~\ref{lemma.4t7s.vleq151} except possibly for the values of~$t$ and~$s$ listed above.
Otherwise, if $v \geq 160$, then for each congruence class of~$v$ modulo~$84$, a reference for the relevant $4$-GDD is given in Table~\ref{table.4t7s.vgeq160.summary}.
\end{proof}
\end{theorem}

\begin{table}[hbt!]\small
\caption{Constructions for $4$-GDDs of type~$4^t 7^s$ with $v \geq 160$ points in Theorem~\ref{theorem.4t7s.summary}.}
\label{table.4t7s.vgeq160.summary}
{\noindent
\begin{center}
  \begin{tabular}{|l|l|l|} \hline
    $v \pmod{84}$ & Range of $v$ & Reference  \\ \hline
        
    $4, 7, 16$ & $v \geq 172$ & Lemma~\ref{lemma.4t7s.inf.family.main} \\ \hline
    $19$    & $187 \leq v \leq 523$ & Lemma~\ref{lemma.4t7s.adhoc.rec.v19mod84} \\ \hline
            & $v \geq 607$ & Lemma~\ref{lemma.4t7s.inf.family.v19mod84} \\ \hline
    $28, 31, 40, 43, 52, 55, 64, 67$ & $v \geq 196$ & Lemma~\ref{lemma.4t7s.inf.family.main} \\ \hline
    $76$    & $160 \leq v \leq 496$ & Lemma~\ref{lemma.4t7s.adhoc.rec.v76mod84} \\ \hline
            & $v \geq 580$ & Lemma~\ref{lemma.4t7s.inf.family.v76mod84} \\ \hline
    $79$    & $163 \leq v \leq 247$ & Lemma~\ref{lemma.4t7s.adhoc.rec.v79mod84} \\ \hline
            & $v \geq 331$ & Lemma~\ref{lemma.4t7s.inf.family.v79mod84} \\ \hline
  \end{tabular}
\end{center}}
\end{table}


By forming a block on each group of size $4$ in the known $4$-GDDs in Theorem~\ref{theorem.4t7s.summary}, we get Corollary~\ref{corollary.4gdd.1^4t_7^s.from.4t7s}.

\begin{corollary}
\label{corollary.4gdd.1^4t_7^s.from.4t7s}
Suppose that $t+s \equiv 1 \pmod{3}$ and $s \equiv 0$ or $1 \pmod{4}$ where either $t \geq 4$ or $s \geq 4$. 
Then there exists a $4$-GDD of type $1^{4t} 7^s$ except possibly for the pairs $(t,s)$ for which a $4$-GDD of type~$4^t 7^s$ is given as unknown in Theorem~\ref{theorem.4t7s.summary}.
\end{corollary}

\section{Acknowledgements} 
This research used the computational cluster Katana supported by Research Technology Services at UNSW Sydney.
The third author acknowledges the support from an Australian Government Research Training Program Scholarship and from the School of Mathematics and Statistics, UNSW Sydney.
The authors are grateful to the anonymous referee for providing the construction in Lemma~\ref{lemma.4gdd.4^16_7^9}.

\section*{ORCID}
R. J. R. Abel:     https://orcid.org/0000-0002-3632-9612\\
T. Britz:          https://orcid.org/0000-0003-4891-3055\\
Y. A. Bunjamin:    https://orcid.org/0000-0001-6849-2986\\
D. Combe:          https://orcid.org/0000-0002-1055-3894

\section*{Appendix}

\begin{table}[ht]
\caption{$4$-GDD of type $4^6 7^4$} \label{4674}
\medskip
Points: $a_i,b_i,c_i,f_i,p_i,q_i,r_i,s_i$ for $i\in\Z_6$; \, $d_i$ for $i\in\Z_3$; \, $\infty$.\\
Groups: $\{a_i, a_{i+3}, b_i, b_{i+3}, c_i, c_{i+3}, d_i\}$ for $i \in \{0,1,2\}$; \, $\{f_i : i \in \Z_6\}\cup\{\infty\}$;\\\phantom{Groups:}\hspace*{.33mm}
        $\{p_i, p_{i+3}, q_i, q_{i+3}\}$ and $\{r_i, r_{i+3}, s_i, s_{i+3}\}$ for $i \in \{0,1,2\}$.\\
Develop the following blocks  $(\bmod\ 6)$. 
The first block in the first column generates just a single block.
The second block in the first column generates two blocks.
{\small
\[ 
  \begin{array}{|l|l|l|l|l|l|}\hline
    \{d_0,d_1,d_2,\infty\} & \{a_0,a_1,f_0,p_0\} & \{a_0,c_5,q_4,s_2\} & \{b_0,b_1,f_3,r_4\} & \{b_0,p_4,s_5,\infty\} & \{c_0,q_1,r_1,\infty\}\\\hline
    \{a_0,a_2,a_4,\infty\} & \{a_0,b_1,c_2,r_0\} & \{a_0,d_1,f_4,p_3\} & \{b_0,c_5,q_1,s_3\} & \{c_0,c_2,f_0,p_2\} & \{d_0,p_1,q_2,r_5\}\\\hline
     & \{a_0,b_2,f_1,q_0\} & \{a_0,d_2,q_5,s_5\} & \{b_0,d_2,f_4,q_0\} & \{c_0,c_1,f_3,r_0\} & \{f_0,p_3,p_4,s_2\}\\\hline
     & \{a_0,b_4,p_1,q_3\} & \{a_0,f_3,p_4,r_3\} & \{b_0,d_1,p_1,r_1\} & \{c_0,d_1,f_5,s_2\} & \{f_0,q_0,q_4,s_1\}\\\hline
     & \{a_0,b_5,r_1,s_0\} & \{a_0,p_2,r_4,r_5\} & \{b_0,f_1,q_2,r_0\} & \{c_0,d_2,r_3,s_1\} & \{f_0,q_3,r_2,r_4\}\\\hline
     & \{a_0,c_1,f_2,s_1\} & \{a_0,r_2,s_3,s_4\} & \{b_0,f_0,s_0,s_4\} & \{c_0,p_3,p_5,s_5\} & \\\hline
     & \{a_0,c_4,q_1,q_2\} & \{b_0,b_2,c_4,p_2\} & \{b_0,p_5,q_3,s_2\} & \{c_0,p_1,q_0,r_2\} & \\\hline
  \end{array}
\]}
\end{table}

\begin{table}[ht]
\caption{$4$-GDD of type $4^5 7^5$} \label{4575}
\medskip
Points: $a_i,b_i,c_i,d_i,e_i,f_i,g_i,p_i,q_i,r_i,s_i$ for $i\in\Z_5$.\\
Groups: $\{a_i,b_i,c_i,d_i,e_i,f_i,g_i\}$ for $i \in \Z_5$;  \, 
        $\{p_i,q_i,r_i,s_i\}$ for $i\in \Z_5$.\\
Develop the following blocks  $(\bmod\ 5)$.
{\small
\[
  \begin{array}{|l|l|l|l|l|l|}\hline
    \{a_0,a_1,b_2,q_0\} & \{a_0,c_3,r_3,s_0\} & \{b_0,b_1,e_4,q_0\} & \{b_0,f_2,p_2,r_1\} & \{c_0,f_1,r_2,r_3\} & \{e_0,e_1,f_2,p_0\}\\\hline
    \{a_0,a_2,p_0,s_1\} & \{a_0,d_1,d_4,f_2\} & \{b_0,b_2,f_3,g_1\} & \{b_0,g_2,p_0,r_3\} & \{c_0,g_1,r_1,s_4\} & \{e_0,f_4,s_2,s_4\}\\\hline
    \{a_0,b_3,c_1,r_0\} & \{a_0,d_2,d_3,g_4\} & \{b_0,c_2,d_4,q_1\} & \{b_0,p_3,p_4,s_2\} & \{d_0,e_1,e_3,g_4\} & \{e_0,g_2,g_4,s_1\}\\\hline
    \{a_0,b_4,e_1,q_1\} & \{a_0,e_2,p_4,r_1\} & \{b_0,c_4,d_3,s_0\} & \{c_0,c_2,e_1,p_2\} & \{d_0,e_2,q_0,r_2\} & \{f_0,f_1,g_2,q_0\}\\\hline
    \{a_0,c_2,e_4,s_2\} & \{a_0,e_3,r_4,s_3\} & \{b_0,c_1,g_3,s_4\} & \{c_0,c_1,f_4,q_2\} & \{d_0,e_4,q_3,r_1\} & \{f_0,q_1,q_2,s_4\}\\\hline
    \{a_0,c_4,g_2,q_2\} & \{a_0,f_1,f_3,p_2\} & \{b_0,d_2,p_1,s_3\} & \{c_0,d_3,f_2,p_4\} & \{d_0,f_2,s_3,s_4\} & \{g_0,g_1,p_2,r_3\}\\\hline
                        & \{a_0,f_4,g_3,r_2\} & \{b_0,d_1,r_0,s_1\} & \{c_0,d_1,g_4,p_3\} & \{d_0,p_0,p_3,q_1\} & \{g_0,q_1,q_4,s_0\}\\\hline
                        & \{a_0,g_1,p_1,q_3\} & \{b_0,e_1,f_4,r_4\} & \{c_0,e_3,p_1,q_0\} & \{d_0,q_4,r_0,r_3\} & \\\hline
  \end{array}
\]}
\end{table}

\newpage

\begin{table}[ht]
\caption{$4$-GDD of type $4^9 7^4$} \label{4974}
\medskip
Points: $a_i,b_i,c_i,d_i,p_i,q_i,r_i,s_i,y_i,z_i$ for $i\in\Z_6$; \, $\infty_1$, $\infty_2$, $\infty_3$, $\infty_4$.\\
Groups: $\{a_i : i\in\Z_6\}\cup\{\infty_1\}$; \,  $\{b_i : i\in\Z_6\}\cup\{\infty_2\}$; \, $\{c_i : i\in\Z_6\}\cup\{\infty_3\}$; \, $\{d_i : i\in\Z_6\}\cup\{\infty_4\}$;\\\phantom{Groups:}\hspace*{.33mm}
        $\{p_i, p_{i+3}, q_i, q_{i+3}\}$, $\{r_i, r_{i+3}, s_i, s_{i+3}\}$ and $\{y_i, y_{i+3}, z_i, z_{i+3}\}$ for $i \in \{0,1,2\}$.\\
Develop the following blocks  $(\bmod\ 6)$. The block in the first column generates just a single block.
{\small
\[ 
  \begin{array}{|l|l|l|l|l|l|}\hline
    \{\infty_1,\infty_2,\infty_3,\infty_4\} & \{a_0,b_0,d_0,p_0\} & \{a_0,c_5,z_3,\infty_2\} & \{b_0,c_4,q_3,y_0\}      & \{c_0,d_5,y_5,\infty_1\} & \{d_0,r_2,y_3,z_4\}\\\hline
     & \{a_0,b_1,d_2,y_0\}                  & \{a_0,d_4,q_2,r_2\} & \{b_0,c_5,r_0,z_4\}      & \{c_0,p_3,p_5,q_1\}      & \{d_0,s_2,y_5,\infty_2\}\\\hline
     & \{a_0,b_2,p_1,y_3\}                  & \{a_0,d_5,r_4,s_3\} & \{b_0,c_2,s_3,y_2\}      & \{c_0,p_4,r_4,r_5\}      & \{p_0,q_1,r_5,\infty_2\}\\\hline
     & \{a_0,b_3,r_0,z_0\}                  & \{a_0,p_3,p_4,z_4\} & \{b_0,d_2,p_3,y_4\}      & \{c_0,q_2,y_3,\infty_4\} & \{p_0,q_5,s_1,\infty_1\}\\\hline
     & \{a_0,b_4,s_0,\infty_4\}             & \{a_0,p_5,y_2,y_4\} & \{b_0,d_4,q_5,r_1\}      & \{c_0,s_0,y_1,z_3\}      & \{p_0,r_3,y_0,\infty_3\}\\\hline
     & \{a_0,b_5,z_1,\infty_3\}             & \{a_0,q_3,q_4,z_5\} & \{b_0,p_2,s_0,z_1\}      & \{c_0,s_2,z_0,z_2\}      & \{p_0,r_2,z_3,\infty_4\}\\\hline
     & \{a_0,c_0,d_1,q_0\}                  & \{a_0,q_5,s_4,s_5\} & \{b_0,p_4,s_1,z_0\}      & \{d_0,p_3,s_3,y_1\}      & \{q_0,y_4,y_5,z_3\}\\\hline
     & \{a_0,c_1,d_3,z_2\}                  & \{a_0,r_3,r_5,y_5\} & \{b_0,q_2,q_4,s_5\}      & \{d_0,p_5,s_1,z_3\}      & \{r_0,s_2,s_4,y_4\}\\\hline
     & \{a_0,c_2,p_2,s_1\}                  & \{b_0,c_0,d_3,p_1\} & \{b_0,q_1,r_4,y_3\}      & \{d_0,q_0,r_1,z_0\}      & \\\hline
     & \{a_0,c_3,q_1,y_1\}                  & \{b_0,c_1,d_5,s_4\} & \{b_0,r_2,z_5,\infty_1\} & \{d_0,q_2,s_0,\infty_3\} & \\\hline
     & \{a_0,c_4,r_1,s_2\}                  & \{b_0,c_3,q_0,r_5\} & \{c_0,d_0,p_2,r_0\}      & \{d_0,q_3,z_1,z_2\}      & \\\hline
  \end{array}
\]}
\end{table}

\begin{table}[ht]
\caption{$4$-GDD of type $4^2 7^8$} \label{4278}
\medskip
Points: $a_i,b_i,c_i,d_i,e_i,f_i,g_i,p_i,x_i,y_i$ for $i\in\Z_6$; \, $q_i,z_i$ for $i\in\Z_2$.\\
Groups: $\{a_i,b_i,c_i,d_i,e_i,f_i,g_i\}$ for $i \in \Z_6$;\\\phantom{Groups:}\hspace*{.33mm}
        $\{x_i,x_{i+2},x_{i+4},y_i,y_{i+2},y_{i+4},z_i\}$ and $\{p_i,p_{i+2},p_{i+4},q_i\}$ for $i\in \{0,1\}$.\\
Develop the following blocks  $(\bmod\ 6)$.
The first block in the first column generates just a single block.
The second block in the first column generates two blocks.
All other blocks in the first column generate three blocks each.
{\small
\[
  \begin{array}{|l|l|l|l|l|l|}\hline
    \{q_0,q_1,z_0,z_1\} & \{a_0,a_1,f_2,p_0\} & \{a_0,d_5,e_1,x_3\} & \{b_0,d_3,d_5,g_4\} & \{c_0,c_2,g_5,x_2\} & \{d_0,f_4,p_5,y_2\}\\\hline
    \{a_0,a_2,a_4,q_0\} & \{a_0,b_2,c_5,y_0\} & \{a_0,e_3,g_5,z_1\} & \{b_0,d_2,p_5,x_1\} & \{c_0,d_3,f_2,p_4\} & \{d_0,g_3,q_0,x_1\}\\\hline
    \{a_0,a_3,b_1,b_4\} & \{a_0,b_3,f_4,x_5\} & \{a_0,e_2,p_3,y_4\} & \{b_0,d_4,p_0,x_0\} & \{c_0,d_5,p_3,z_1\} & \{d_0,g_2,y_4,z_1\}\\\hline
    \{c_0,c_3,d_1,d_4\} & \{a_0,b_5,p_1,z_0\} & \{a_0,e_4,p_4,y_3\} & \{b_0,d_1,q_0,y_0\} & \{c_0,e_4,f_3,q_0\} & \{e_0,f_4,p_4,x_3\}\\\hline
    \{e_0,e_3,y_0,y_3\} & \{a_0,c_1,d_3,x_0\} & \{a_0,f_3,g_4,p_2\} & \{b_0,e_2,e_4,g_3\} & \{c_0,e_5,f_1,y_0\} & \{e_0,f_1,q_1,y_4\}\\\hline
    \{f_0,f_3,x_0,x_3\} & \{a_0,c_2,g_3,q_1\} & \{a_0,f_5,g_1,x_1\} & \{b_0,e_5,x_3,x_4\} & \{c_0,e_3,x_3,z_0\} & \{f_0,g_4,p_3,y_1\}\\\hline
    \{g_0,g_3,p_0,p_3\} & \{a_0,c_3,x_4,y_5\} & \{b_0,b_1,c_5,f_3\} & \{b_0,f_4,f_5,z_0\} & \{c_0,f_5,g_2,g_4\} & \{g_0,g_1,x_2,y_5\}\\\hline
                        & \{a_0,c_4,x_2,y_1\} & \{b_0,b_2,e_3,g_1\} & \{b_0,g_2,y_2,y_3\} & \{d_0,e_3,e_4,p_0\} & \{g_0,p_1,p_2,x_5\}\\\hline
                        & \{a_0,d_1,d_2,y_2\} & \{b_0,c_1,p_3,y_5\} & \{b_0,p_4,q_1,x_5\} & \{d_0,e_5,f_2,x_0\} & \\\hline
                        & \{a_0,d_4,e_5,g_2\} & \{b_0,c_2,p_1,y_1\} & \{c_0,c_1,e_2,p_1\} & \{d_0,f_1,f_3,y_3\} & \\\hline
  \end{array}
\]}
\end{table}

\begin{table}[ht]
\caption{$4$-GDD of type $4^8 7^5$} \label{4875}
\medskip
Points: $a_i,b_i,c_i,d_i,p_i,q_i,r_i,s_i,t_i$ for $i\in\Z_7$; \; $\infty_1,\infty_2,\infty_3,\infty_4$.\\
Groups: $\{a_i,b_i,c_i,d_i\}$ for $i \in \Z_7$; \; $\{w_i\::\: i\in\Z_7\}$ for $w\in \{p,q,r,s,t\}$; \; $\{\infty_1,\infty_2,\infty_3,\infty_4\}$.\\
Develop the following blocks  $(\bmod\ 7)$.
{\small
\[
  \begin{array}{|l|l|l|l|l|l|}\hline
    \{a_0,a_1,b_2,t_0\}      & \{a_0,d_2,d_3,p_1\}      & \{b_0,b_2,c_1,s_1\}      & \{b_0,p_1,q_4,r_5\}      & \{c_0,d_5,q_2,s_5\}      & \{d_0,r_3,s_6,t_2\}\\\hline
    \{a_0,a_2,c_1,s_0\}      & \{a_0,d_5,q_1,r_6\}      & \{b_0,b_3,s_0,t_0\}      & \{b_0,p_2,q_6,s_5\}      & \{c_0,p_6,q_0,t_2\}      & \{d_0,r_0,s_2,\infty_2\}\\\hline
    \{a_0,a_3,q_0,r_0\}      & \{a_0,d_6,s_2,\infty_4\} & \{b_0,c_2,q_0,r_4\}      & \{b_0,p_5,s_2,t_6\}      & \{c_0,p_3,q_1,\infty_4\} & \{p_0,r_1,s_5,t_6\}\\\hline
    \{a_0,b_3,b_4,p_0\}      & \{a_0,p_2,q_2,s_3\}      & \{b_0,c_3,q_2,t_3\}      & \{b_0,q_5,r_0,\infty_3\} & \{c_0,q_4,r_3,s_1\}      & \{q_0,r_3,s_2,t_4\}\\\hline
    \{a_0,b_5,c_2,d_1\}      & \{a_0,p_3,q_5,t_3\}      & \{b_0,c_5,q_1,\infty_2\} & \{b_0,r_6,t_1,\infty_4\} & \{c_0,r_5,t_1,\infty_1\} & \\\hline
    \{a_0,b_6,r_1,t_1\}      & \{a_0,p_4,q_3,\infty_1\} & \{b_0,d_1,d_5,q_3\}      & \{c_0,c_1,d_4,t_4\}      & \{d_0,d_2,q_1,t_1\}      & \\\hline
    \{a_0,c_3,d_4,r_2\}      & \{a_0,p_5,r_3,s_4\}      & \{b_0,d_4,p_0,r_3\}      & \{c_0,c_3,p_1,r_0\}      & \{d_0,p_4,r_4,s_4\}      & \\\hline
    \{a_0,c_4,r_5,t_2\}      & \{a_0,p_6,t_4,\infty_2\} & \{b_0,d_6,p_6,r_1\}      & \{c_0,c_2,p_2,s_4\}      & \{d_0,p_1,t_5,\infty_3\} & \\\hline
    \{a_0,c_5,s_1,\infty_3\} & \{a_0,q_6,s_6,t_5\}      & \{b_0,d_2,s_3,\infty_1\} & \{c_0,d_2,p_4,t_6\}      & \{d_0,q_0,s_5,t_3\}      & \\\hline
  \end{array}
\]}
\end{table}

\begin{table}[ht]
\caption{$4$-GDD of type $4^8 7^8$} \label{4878}
\medskip
Points: $a_i,b_i,c_i,d_i,p_i,q_i$ for $i\in\Z_{14}$; \; $\infty_1,\infty_2,\infty_3,\infty_4$.\\
Groups: $\{w_i,w_{i+2},w_{i+4},w_{i+6},w_{i+8},w_{i+10},w_{i+12}\}$ for $w \in \{a,b,c,d\}$ and $i\in\{0,1\}$;\\\phantom{Groups:}\hspace*{.33mm}
        $\{p_i,p_{i+7},q_i,q_{i+7}\}$ for $i\in \{0,1,2,3,4,5,6\}$; \; $\{\infty_1,\infty_2,\infty_3,\infty_4\}$.\\
Develop the following blocks  $(\bmod\ 14)$. The two blocks in the first column generate seven blocks each.
{\small
\[
  \begin{array}{|l|l|l|l|l|l|}\hline
 \{a_0,a_7,c_{1},c_{8}\}    &  \{a_0,a_{1},b_{2},p_{5}\}    &  \{a_0,b_{8},c_{10},q_{2}\}     &  \{a_0,c_{12},p_{11},\infty_2\}    & \{b_0,c_{5},p_{0},q_{2}\}       &  \{c_{0},d_{7},q_{3},q_{13}\}  \\\hline
 \{b_0,b_7,d_{1},d_{8}\}    &  \{a_{0},a_{3},d_{1},p_{2}\}  &  \{a_{0},b_{9},c_{13},q_{8}\}   &  \{a_{0},d_{10},p_{10},q_{13}\}    & \{b_{0},c_{7},p_{13},\infty_3\} &  \{c_{0},d_{9},q_{2},q_{4}\}\\\hline
                            &  \{a_{0},a_{5},d_{2},q_{0}\}  &  \{a_0,b_{11},d_{3},p_{9}\}     &  \{a_0,d_{7},q_{12},\infty_3\}     & \{b_{0},c_{8},q_{1},q_{4}\}     &  \{c_0,d_{3},q_{5},\infty_1\} \\\hline
                            &  \{a_{0},b_{0},c_{0},d_{0}\}  &  \{a_{0},b_{12},d_{5},q_{4}\}   &  \{a_{0},p_{0},p_{1},p_{3}\}       & \{b_{0},d_{5},q_{9},\infty_2\}  &  \{d_{0},d_{3},p_{5},p_{11}\}\\\hline
                            &  \{a_{0},b_{3},c_{2},d_{6}\}  &  \{a_{0},b_{13},d_{9},q_{6}\}   &  \{a_{0},p_{12},q_{10},q_{11}\}    & \{b_{0},p_{8},q_{0},q_{5}\}     &  \{d_{0},d_{1},p_{10},q_{1}\}\\\hline
                            &  \{a_{0},b_{4},c_{5},p_{6}\}\ & \{a_0,b_{10},p_{7},\infty_1\}   &  \{b_{0},b_{3},c_{9},p_{7}\}       & \{b_{0},p_{9},q_{3},\infty_4\}  &  \{p_{0},p_{5},q_{1},q_{9}\}\\\hline
			          &  \{a_{0},b_{5},c_{3},q_{1}\}  & \{a_{0},c_{6},d_{4},p_{8}\}     &  \{b_{0},b_{1},d_{13},p_{6}\}      & \{c_{0},c_{1},d_{2},p_{5}\}  & \\\hline
                            &  \{a_{0},b_{6},c_{9},q_{3}\}  & \{a_{0},c_{7},d_{13},q_{7}\}    &  \{b_{0},b_{5},d_{2},p_{1}\}       & \{c_{0},c_{5},d_{10},p_{8}\} & \\\hline
                            &  \{a_{0},b_{7},c_{4},q_{5}\}  & \{a_{0},c_{11},d_{8},\infty_4\} &  \{b_{0},c_{10},d_{4},d_{9}\}      & \{c_{0},c_{3},p_{0},p_{10}\} & \\\hline
  \end{array}
\]}
\end{table}

\begin{table}[ht]
\caption{$4$-GDD of type $4^{14} 7^5$} \label{41475}
\medskip
Points: $a_i,b_i,c_i,d_i,p_i,q_i$ for $i\in\Z_{14}$; \; $r_i,s_i$ for $i\in\Z_2$; \; $\infty_1,\infty_2,\infty_3$.\\
Groups: $\{a_i,a_{i+7},b_i,b_{i+7}\}$ and $\{c_i,c_{i+7},d_i,d_{i+7}\}$ for $i\in \{0,1,2,3,4,5,6\}$;\\\phantom{Groups:}\hspace*{.33mm}
        $\{w_i,w_{i+2},w_{i+4},w_{i+6},w_{i+8},w_{i+10},w_{i+12}\}$ for $w \in \{p,q\}$ and $i\in\{0,1\}$;\\\phantom{Groups:}\hspace*{.33mm}
        $\{r_0,r_1,s_0,s_1,\infty_1,\infty_2,\infty_3\}$.\\
Develop the following blocks  $(\bmod\ 14)$. The block in the first column generates seven blocks.
{\small
\[
  \begin{array}{|l|l|l|l|l|l|}\hline
    \{p_0,p_{7},q_{1},q_{8}\} & \{a_0,a_{1},a_{3},p_0\}       &  \{a_0,b_{9},c_{5},q_{6}\}      & \{a_0,d_{7},p_{10},q_{10}\}     & \{b_0,c_{11},p_{8},\infty_3\} & \{c_0,d_{8},p_{3},q_{10}\}\\\hline
                              &  \{a_0,a_{4},b_{2},p_{2}\}    &  \{a_0,b_{11},c_{6},q_{13}\}    & \{a_0,d_{11},q_{8},r_{1}\}      & \{b_0,d_{1},p_{12},q_{7}\}    & \{c_0,d_{13},p_{5},\infty_1\}\\\hline
                              &  \{a_0,a_{5},d_0,q_0\}        &  \{a_0,b_{10},d_{3},r_0\}       &  \{a_0,d_{12},q_{7},s_{1}\}     & \{b_0,d_{3},p_{2},q_{1}\}     & \{c_0,d_{4},q_0,q_{11}\}\\\hline
                              &  \{a_0,a_{6},d_{2},q_{3}\}    &  \{a_0,b_{13},q_{5},\infty_1\}  &  \{a_0,d_{13},q_{12},\infty_3\} & \{b_0,d_{10},p_{4},q_0\}      & \{c_0,p_{2},p_{7},q_{13}\}\\\hline
                              &  \{a_0,b_{3},c_{3},c_{4}\}    &  \{a_0,c_{8},d_{4},d_{5}\}      &  \{b_0,b_{2},b_{6},p_{7}\}      & \{b_0,d_{9},q_{3},\infty_2\}  & \{d_0,d_{3},d_{9},p_{7}\}\\\hline
                              &  \{a_0,b_{1},c_0,d_{1}\}      &  \{a_0,c_{10},d_{8},p_{4}\}     & \{b_0,b_{3},c_{8},s_0\}         & \{b_0,p_{6},q_{4},q_{9}\}     & \{d_0,p_0,q_{5},s_0\}\\\hline
                              &  \{a_0,b_{4},c_{2},p_{3}\}    &  \{a_0,c_{11},p_{6},p_{9}\}     & \{b_0,b_{5},d_{2},d_{4}\}       & \{b_0,p_{9},q_{13},r_{1}\}    & \\\hline
                              &  \{a_0,b_{5},c_{7},q_{1}\}    &  \{a_0,c_{13},p_{5},s_0\}       & \{b_0,b_{1},d_{6},p_{11}\}      & \{c_0,c_{2},c_{10},d_{5}\}    & \\\hline
                              &  \{a_0,b_{6},c_{9},q_{4}\}    &  \{a_0,c_{1},p_{1},\infty_2\}   & \{b_0,c_{6},d_{8},d_{12}\}      & \{c_0,c_{5},p_{4},r_{1}\}     & \\\hline
                              &  \{a_0,b_{8},c_{12},q_{2}\}   & \{a_0,d_{6},p_{7},p_{8}\}       & \{b_0,c_{7},p_{3},q_{5}\}       & \{c_0,c_{3},q_{5},q_{6}\}     & \\\hline
  \end{array}
\]}
\end{table}

\begin{table}[ht]
\small
\caption{$4$-GDD of type $1^{64}7^{9}$} 
\label{16479}
\medskip
Points: $a_i,b_i,c_i,p_i,q_i,r_i$ for $i\in Z_{21}$; \, $\infty$.   \\
Groups: $\{w_i,w_{i+3},w_{i+6}, \ldots, w_{i+18}\}$   for $w \in \{a,b,c\}$ and $i \in \{0,1,2\}$;   \\\phantom{Groups:}\hspace*{.33mm}
$\{w_i\}$   for $w \in \{p,q,r\}$ and $i \in Z_{21}$; \, $\{\infty\}$.       

Develop the following blocks  $(\bmod\ 21)$. 
{\small
\[ 
  \begin{array}{|l|l|l|l|l|}\hline
  \{a_{0},a_{7},b_{17},q_{5}\}    & \{a_{0},b_{4},c_{9},r_{19}\}     & \{a_{0},c_{7},p_{7},p_{16}\}    & \{b_{0},c_{2},c_{16},r_{9}\}    &  \{p_{0},p_{11},q_{0},r_{0}\}     \\\hline 
  \{a_{0},a_{8},c_{4},p_{6}\}     & \{a_{0},b_{7},c_{20},r_{3}\}     & \{a_{18},c_{11},p_{6},p_{19}\}  & \{b_{0},c_{1},p_{5},p_{8}\}     &  \{p_{4},p_{18},q_{17},r_{0}\}     \\\hline
  \{a_{0},a_{16},c_{3},q_{3}\}    & \{a_{0},b_{18},c_{5},\infty\}    & \{a_{1},c_{19},p_{12},q_{18}\}  & \{b_{0},c_{4},p_{9},p_{15}\}    &  \{p_{12},p_{17},q_{8},q_{10}\}     \\\hline 
  \{a_{0},a_{4},p_{4},r_{8}\}     & \{a_{16},b_{3},p_{15},p_{19}\}   & \{a_{0},c_{1},p_{2},r_{13}\}    & \{b_{0},c_{6},p_{19},q_{16}\}   &  \{p_{8},p_{9},q_{13},r_{10}\}     \\\hline
  \{a_{0},a_{20},q_{11},q_{20}\}  & \{a_{4},b_{15},p_{19},q_{14}\}   & \{a_{18},c_{18},q_{12},r_{2}\}  & \{b_{4},c_{16},p_{1},r_{17}\}   &  \{p_{11},q_{5},q_{19},r_{3}\}     \\\hline 
  \{a_{0},a_{2},q_{6},r_{18}\}    & \{a_{10},b_{5},p_{7},r_{19}\}     & \{a_{11},c_{2},q_{18},r_{17}\}  & \{b_{7},c_{10},p_{13},r_{19}\}  &  \{p_{19},q_{9},q_{20},r_{16}\}     \\\hline
  \{a_{9},a_{20},r_{16},r_{20}\}  & \{a_{7},b_{10},p_{20},r_{8}\}    & \{a_{0},c_{2},r_{2},r_{20}\}    & \{b_{0},c_{20},q_{4},q_{12}\}   &  \{p_{20},q_{8},r_{14},\infty\}     \\\hline 
  \{a_{0},b_{2},b_{15},p_{5}\}    & \{a_{0},b_{1},p_{14},r_{12}\}    & \{b_{0},b_{1},c_{18},q_{15}\}   & \{b_{0},c_{7},q_{13},r_{5}\}    &  \{p_{10},r_{9},r_{15},r_{17}\}     \\\hline
  \{a_{0},b_{9},b_{14},q_{14}\}   & \{a_{0},b_{0},p_{17},r_{10}\}    & \{b_{2},b_{4},p_{3},q_{10}\}    & \{b_{0},c_{0},q_{11},r_{20}\}   &  \{q_{14},q_{15},q_{19},r_{8}\}     \\\hline 
  \{a_{0},b_{5},c_{15},p_{12}\}   & \{a_{0},b_{6},q_{13},q_{16}\}    & \{b_{7},b_{14},p_{7},r_{15}\}   & \{c_{16},c_{18},p_{12},q_{14}\} &  \{q_{10},q_{16},r_{11},r_{18}\}     \\\hline
  \{a_{0},b_{19},c_{13},q_{1}\}   & \{a_{6},b_{18},r_{0},r_{20}\}    & \{b_{0},b_{4},q_{1},r_{4}\}     & \{c_{0},c_{1},p_{20},q_{2}\}    &      \\\hline 
  \{a_{0},b_{20},c_{10},q_{18}\}  & \{a_{0},c_{11},c_{16},q_{2}\}    & \{b_{7},b_{18},r_{4},r_{13}\}   & \{c_{9},c_{17},r_{4},r_{20}\}   &      \\\hline
  \{a_{0},b_{13},c_{6},q_{9}\}    & \{a_{0},c_{19},p_{8},p_{10}\}    & \{b_{0},c_{9},c_{19},q_{2}\}    & \{c_{10},c_{14},r_{6},r_{16}\}  &     \\\hline
 \end{array}
\]}
\end{table}


\begin{thebibliography}{99}



\bibitem{ABBC.2t8s}
R.J.R. Abel, T. Britz, Y.A. Bunjamin and  D. Combe, Group divisible designs with block size $4$ where the group sizes are congruent to $2 \pmod{3}$, Discrete Math. 345 (2022), 112740.

\bibitem{ABBC.2t5s}
R.J.R. Abel, T. Britz, Y.A. Bunjamin and  D. Combe, Group divisible designs with block size $4$ and group sizes $2$ and $5$,  J. Combin. Des. 30 (2022), 367--383.  

\bibitem{ABBCF.1mod3}
R.J.R. Abel, T. Britz, Y.A. Bunjamin, D. Combe and T. Feng, Group divisible designs with block size 4 where the group sizes are congruent to $1 \pmod{3}$, Discrete Math. 346 (2023), 113277.

\bibitem{ABC.30less}
R.J.R. Abel,  Y.A. Bunjamin and D. Combe, Some new group divisible designs with block size $4$ and two or three group sizes,  J. Combin. Des. 28 (2020), 614--628.  

\bibitem{ABC.50less}
R.J.R. Abel, Y.A. Bunjamin and  D. Combe, Existence of 4-GDDs with at most 50 points and $4$-GDDs of types $6^s 3^t$ and $9^s 3^t$, Discrete Math. 344 (2021), 112479.

\bibitem{ABC.3562}
R.J.R. Abel,  Y.A. Bunjamin and D. Combe, The 4-GDDs of type $3^5 6^2$, Discrete Math. 345 (2022), 112983.




\bibitem{BSH}
A.E. Brouwer, A. Schrijver and H. Hanani,  Group divisible designs with block size 4, Discrete Math. 20 (1977), 1--10.



\bibitem{Forbes2}
A.D. Forbes,  Group divisible designs with block size 4 and type $g^u m^1$ II,  J. Combin. Des.  27 (2019), 311--349.

\bibitem{Forbes3}
A.D. Forbes,  Group divisible designs with block size 4 and type $g^u m^1$ III,  J. Combin. Des.  27 (2019), 643--672.

\bibitem{Forbesgcc}
A.D. Forbes,  Group divisible designs with block size 4 and type $g^u b^1 (gu/2)^1$,  Graphs Combin. 36 (2020), 1687--1703. 

\bibitem{Forbes1}
A.D. Forbes and K.A. Forbes,  Group divisible designs with block size 4 and type $g^u m^1$,  J. Combin. Des.  26 (2018), 519--539.

\bibitem{fmy}   
S.C. Furino, Y. Miao and J. Yin, Frames and resolvable designs, CRC Press, Boca Raton FL, 1996.

\bibitem{GeLing.2004.gum1}
G. Ge and A. Ling, Group divisible designs with block size four and group type $g^u m^1$ for small $g$, Discrete Math. 285 (2004), 97--120.

\bibitem{gerees}
G. Ge and R.S. Rees, On group-divisible designs with block size four and group type $g^u m^1$, Des. Codes Cryptogr. 27 (2002), 5--24.


\bibitem{gezhu}
G. Ge, R.S. Rees and L. Zhu, Group divisible designs with block size 4 and group type $g^u m^1$ with $m$ as small as possible,  J. Combin. Theory Ser. A  98 (2002), 357-376.  

\bibitem{Hanani.1975}
H. Hanani, Balanced incomplete block designs and related designs, Discrete Math. 11 (1975) 255–369.



\bibitem{kreher.210note}
D.L. Kreher, A.C.H. Ling, R.S. Rees and C.W.H. Lam, A note on $\{4\}$-GDDs of type $2^{10}$, Discrete Math. 261 (2003) 373–376.

\bibitem{krestin}
D. Kreher and D.R. Stinson, Small group divisible designs with block size four, J. Statist. Plann. Inference 58 (1997), 111--118.


\bibitem{reesk=45}
R.S. Rees, Group-divisible designs with block size $k$ having $k+1$ groups for $k=4,5$, J. Combin. Des. 8 (2000), 363--386. 





\bibitem{Schuster.gum1.mult8}
E. Schuster, Group divisible designs with block size four and group type $g^u m^1$ where $g$ is a multiple of $8$, Discrete Math. 310 (2010), 2258--2270.

\bibitem{schuster.2014.gum1}
E. Schuster, New classes of group divisible designs with block size 4 and group type $g^u m^1$, J. Combin. Math. Combin. Comput. 91 (2014), 65--105.




\bibitem{WeiGe.2013.more.gum1}
H. Wei and G. Ge, Group divisible designs with block size four and group type $g^u m^1$ for more small $g$, Discrete Math. 313 (2013), 2065--2083.

\bibitem{WeiGe.4gdd.gpn1.g0mod6}
H. Wei and G. Ge, Group divisible designs with block size four and group type $g^u m^1$ for $g \equiv 0 \pmod{6}$, J. Combin. Des. 22 (2014), 26--52.


\bibitem{gegum}
H. Wei and G. Ge, Group divisible designs with block size 4 and type $g^u m^1$, Des. Codes Cryptogr. 74 (2015), 243--282.

\bibitem{Wilson}
R.M. Wilson, Constructions and uses of pairwise balanced designs, Math. Centre Tracts  55 (1974), 18--41.

\end{thebibliography}
\end{document}